\documentclass[11pt, notitlepage]{article}
\usepackage{amssymb,amsmath,comment,esint}
\catcode`\@=11 \@addtoreset{equation}{section}
\def\thesection{\arabic{section}}

\def\theequation{\thesection.\arabic{equation}}
\catcode`\@=12
\usepackage{colortbl}%

\usepackage[mathscr]{eucal}
\usepackage{xcolor}
\usepackage{hyperref}
\usepackage{epsf}
\usepackage{a4wide}
\usepackage[left=1.2 in,right=1.2 in,top=1 in,bottom=.9 in]{geometry}
\def\R{\mathbb{R}}
\newcommand{\EQ}[1]{\eqref{#1}}\newcommand{\re}[1]{\eqref{#1}}

\newcommand{\ds} {\displaystyle}
\newcommand{\e}{\epsilon}

\newcommand{\al} {\alpha}

\newcommand{\de} {\delta}
\newcommand{\ga} {\gamma}
\newcommand{\Ga} {\Gamma}
\newcommand{\Om} {\Omega}

\newcommand{\De} {\Delta}

\newcommand{\noi} {\noindent}

\newcommand{\mb} {\mathbb}
\newcommand{\mc} {\mathcal}

\usepackage[all]{xy}
\catcode`\@=11
\def\theequation{\@arabic{\c@section}.\@arabic{\c@equation}}
\catcode`\@=12

\def\QED{\hfill {$\square$}\goodbreak \medskip}

\newtheorem{Theorem}{Theorem}[section]
\newtheorem{Lemma}[Theorem]{Lemma}
\newtheorem{Proposition}[Theorem]{Proposition}

\newtheorem{Remark}[Theorem]{Remark}
\newtheorem{Definition}[Theorem]{Definition}
\newtheorem{Example}{Example}

\begin{document}
	
	{\vspace{0.01in}}
	
	\title
	{ \sc Nonexistence of positive supersolutions for semilinear fractional elliptic equations in exterior domains}
	
	\author{
		~~ Reshmi Biswas\footnote{e-mail: reshmi15.biswas@gmail.com} ~ and Alexander Quaas \footnote{alexander.quaas@usm.cl}\\
		Departamento de Matem\'atica,\\ Universidad T\'ecnica Federico Santa Mar\'ia\\ Casilla V-110, Avda. Espana,\\ 1680-Valpara\'iso, Chile.
	}
	
	\date{}

	\maketitle
	
	\begin{abstract}
		
		This article is concerned 	with  the nonexistence of positive solutions for  nonlinear fractional elliptic inequalities in exterior domains of $\mathbb R^n$, $n\geq1$. 
		We would like to highlight the fact that our results are new with very weak assumption on the nonlinear term appearing in the inequalities. 
		The results are also novel in this generality 
		even if  the whole space is considered.  
		
		\medskip
		
		\noi \textbf{Key words:} Liouville theorem, exterior domain,  semilinear inequality, fractional Laplacian
		
		\medskip
		
		\noi \textit{2010 Mathematics Subject Classification:} 35R11, 35R09, 35B53, 35D30, 35D40

	\end{abstract}
	
	\section{Introduction }
	In the present article, we study Liouville type theorems  and we are interested in the  nonexistence of positive solutions  of the following inequality:
	\begin{equation}\label{modeleq}
		(-\Delta)^s u \ge f(u,x),
	\end{equation}
	in  exterior domains of $\mb R^n,\, n\geq 1,$ where $s\in (0,1)$ and  $f$ is a positive continuous function with some appropriate assumptions described later.
	Here $ (-\Delta)^s $ is  the fractional Laplacian which is defined, for some suitable function $u:\mb R^n\to \mb R,$ as
	\begin{align*}
		(-\Delta)^s  u(x)&:= C_{n,s} P.V.\int_{\mb R^n}\frac{ u(x)-
			u(y)}{|x-y|^{n+2s}}  dy\notag\\=&C_{n,s}\lim_{\epsilon\to 0^+} \int_{\R^n\setminus B_\epsilon(x) }\frac{ u(x)-
			u(y)}{|x-y|^{n+2s}}  dy,
	\end{align*}
	where P.V. stands for Cauchy´s principal value (as defined by the latter equation), $B_\epsilon(x)$ denotes the ball  centered at $x \in \R^n$ with radius $\epsilon>0$,  and 
	\begin{align*}
		C_{n,s}=2^{2s}\pi^{-\frac n2}s\frac{\Gamma(\frac{n+2s}2)}{\Gamma(1-s)}>0
	\end{align*}   with $\Ga$ being
	the Gamma function. The fractional Laplacian can also be written (after change of variable), up to the normalized constant $C_{n,s}$, as  	\begin{align*}
		(-\Delta)^s  u(x)=C_{n,s}P.V.  \int_{\R^n }\frac{ u(x)-
			u(x-y)}{|y|^{n+2s}}  dy.
	\end{align*} 
	Such operators  are called  integro-differential operators and are nonlocal in nature, which arise  in various fields such as in   the study of thin obstacle problem, 	minimal surfaces, phase
	transitions, conservation laws, optimization, finance,  stratified materials, anomalous diffusion, crystal dislocation, soft thin films, semipermeable membranes, ultra-relativistic limits of quantum mechanics, quasi-geostrophic flows, multiple scattering, etc. For further details on these  operators, we refer to \cite{adv, bisci, hit} and references there in. 
	
	The fundamental solutions for the operator $(-\De)^s$  are radially symmetric (see \cite{CS2, FQ1}), given as, for  $x\in \mb R^n$
	\begin{equation}\label{fs}\Phi(x)=
		\begin{cases}
			|x|^{\sigma^*} \;\;\;\;\; &\text{if } \sigma^*<0\\
			-\log|x| \;\;\;\;\; &\text{if } \sigma^*=0\\
			-|x|^{\sigma^*} \;\;\;\;\; &\text{if } \sigma^*>0,
		\end{cases}
	\end{equation} where $\sigma^*=-n+2s,$ $s\in(0,1)$. We denote $\tilde \Phi:=-\Phi$, which also represents another set of fundamental solutions for $(-\De)^s$.  \\

	The study of elliptic partial differential equations often revolves around a fundamental question: given a nonnegative, nonlinear function \( f := f(u,x) \), under what conditions does there exist a positive solution or supersolution \( u(x) > 0 \) to the equation
	\[ -Q[u] = f(u,x) \]
	in some  domain of \( \mathbb{R}^n \)? Here, \( Q \) represents a second-order elliptic differential operator. 
	A substantial amount  of research developes into the Liouville results for solutions to equations of the type \( -Q[u] = f(u) \) in \( \mathbb{R}^n \), which was initiated by the seminal work of Gidas and Spruck \cite{GS}. For example  and as an extension, it was  established that the equations of type \( -\Delta u = f(u) \) in \( \mathbb{R}^n \) do not attain  classical positive solutions if \( t^{-\frac{n+2}{n-2}}f(t) \) is an increasing function on \( (0,\infty) \), as demonstrated in \cite{Li} and references therein. These findings are profound and intricate, as the range of nonexistence \( \left(\frac{n}{n-2},\frac{n+2}{n-2}\right) \) varies on various factors such as the conformal invariance of the Laplacian, the specific behavior of \( f \) in the whole interval \( (0,\infty) \), the  differential equality in the whole of $\mb R^n$, and the requirement for solutions to be classical in nature.

	\medskip
	In the local frame-work, for the typical semilinear inequality
	\[ (-\Delta)_p u \geq f(u),\;\;  p\geq2,\]
	where \( f \) is a positive continuous function defined on \( (0,+\infty) \), extensive research were conducted to establish the criteria for  nonexistence of positive solutions for such inequalities, both globally in \( \mathbb{R}^n \) and in  subsets of \( \mathbb{R}^n \), encompassing various choices of operators \( Q \) and nonlinear functions \( f \).
	Let us briefly overview the prior findings in this area. Given the extensive literature on both linear and quasilinear scenarios, we do not attempt to compile a comprehensive bibliography here. For thorough references, readers are directed to Veron's book \cite{VB}, and survey articles by Mitidieri and Pohozaev \cite{MP} and Kondratiev, Liskevich, and Sobol [3].
	Gidas \cite{gidas} proved  that the equation 
	$-\Delta u= u^\al$ admits no solutions  in $\R^n$, provided $\al\le \frac{n}{n-2}$.
	Ni and Serrin \cite{NS}, showed  the nonexistence of {\it decaying} radial solutions to some quasilinear inequalities like
	$(-\Delta)_p u\ge |x|^{-\gamma}u^\al$ in $\R^n$ for $\al\le \frac {(n-\gamma)(p-1)}{n-p}$. Later,  Bidaut-Veron \cite{BV} and Bidaut-Veron and  Pohozaev \cite{BVP} expanded these findings by relaxing restrictions on the behavior of a supersolution $u,$ demonstrating their applicability in exterior domains of $\mb R^n$.
	For further nonexistence results concerning positive solutions of quasilinear inequalities with power type nonlinearity in the right-hand sides, see Serrin and Zou \cite{SZ}, Liskevich, Skrypnik, and Skrypnik \cite{LSS}. Liouville-type results for semilinear inequalities in nondivergence form are explored in the work of Kondratiev, Liskevich, and Sobol \cite{KLS}.

	\medskip
	More general result in this aspect was obtained by Armstrong and Sirakov \cite{AS1}, where the authors  studied  the  following differential
	inequality  for $p$-Laplacian $(p\geq 2) $ as well as, for a very general Pucci operator \[ -Q[u] \geq f, \] in an exterior domain $\mb R^n \setminus B$, $n\ge 2$, where $B\subset \R^n$ is any ball. Under the hypotheses that $f:(0,+\infty) \to (0,+\infty)$ is  continuous, as well as
	\begin{equation}\label{modelf}
		0 < \liminf_{t\searrow 0} t^{-\frac {n}{n-2}}f(t) \leq +\infty\qquad\qquad\mbox{ if }n\ge 3,\end{equation}
	\begin{equation}\label{modelf1}
		\liminf_{t\to+\infty} e^{at} f(t) =+ \infty, \quad \mbox{for each} \ a > 0\qquad\mbox{ if } n=2,
	\end{equation}
	for each $a>0$, they established that there does not exist a positive (classical, viscosity or weak) solution of that inequality and they proved the same results with  more general nonlinearity $f(t,x)$ in the same spirit of  the assumptions \eqref{modelf}, \eqref{modelf1}. These types of results were extended to integral conditions by Alarcon, Garcia-Melian and  Quaas in \cite{AGQ}.

	\medskip
	In the the nonlocal frame-work,   Felmer and Quaas \cite{FQ1} first studied  the nonexistence results for the supersolutions of following equation 
	\begin{align}\label{fq}(-\De)^su=u^p \quad\text{in } \mb R^n\end{align}
	for the  critical case $p = \frac {n }{n - 2s}$ as well as  in connection with the extremal operators. {The nonexistence results for this fractional nonlinear equation in the corresponding Sobolev critical ($p= \frac {n+2s }{n - 2s}$) and sub-critical ($p<\frac {n+2s }{n - 2s}$) cases, we refer to  Li \cite{lii},  Chen, Li and Ou \cite{chen} and  other references in this context can be found in \cite{FQ1}.} In a very recent work \cite{chen1}, the authors studied the Liouville theorem for semilinear elliptic inequalities involving the fractional Hardy
	operators.

	\medskip
	Motivated by all the above results, in this article, we aim to establish  Liouville type theorems  for the inequality \EQ{modeleq}. Notably, we relax the conditions on \( f \) required for nonexistence, imposing only "local" criteria on the behaviours of $f(t,x)$ in accordance to the different behaviours of the fundamental solution $\Phi$ (see \eqref{fs}) across the different cases, viz., $2s<n, 2s=n, 2s>n.$ For the case $2s<n$, we consider only the behaviour of $f(t,x)$ near $t=0$  that determines whether or not solutions  exist for the inequality  \eqref{modeleq}, while in the case $2s\geq n$, it is the behavior of $f(t,x)$ at infinity which determines solvability of the inequality \EQ{modeleq}.
	
	Let us give a brief description of our main results in this article (Theorem \ref{thm2},     Theorem \ref{slthm} and Theorem \eqref{thm3}) below. For the convenience of the readers, motivated by \cite{AS1}, consider the example of the nonlinearity $f$ in a simpler form as follows:\begin{equation}\label{gbz}
		f(t,x) = |x|^{-\gamma}  g(t),\quad \ga<2s.
	\end{equation}
	Assume that $s\in(0,1)$, $n\geq 1$ and set $ \widetilde\al^*:= 1+ \frac{2s-\gamma}{-\sigma^*}$, where $\sigma^*:=-n+2s$. Then the  inequality \eqref{modeleq}
	has no positive solution in any exterior domain of $\mb R^n$, provided  the function $g:(0,+\infty)\to(0, +\infty)$ is continuous and satisfies the following:
	\begin{itemize} 
		\item  If $n=1,\, s=\frac 12$, then    $\ds \liminf_{t\to +\infty} e^{bt} g(t) > 0, $ {for every} $b>0.$
		\item  If $n=1,\, s>\frac 12$, {then}  $\ds \liminf_{t\to +\infty} t^{-\widetilde \al^*}g(t) > 0.$
		\item If $n>2s$, then $\ds \liminf_{t\to 0} t^{-\widetilde\al^*}g(t) > 0.$
	\end{itemize}
	We would like to mention  that in this article, we consider  very weak  conditions on $f$ for showing the nonexistence of positive solution for the inequality involving the fractional Laplacian in exterior domains of $\mb R^n,n\geq1$.  Our results also cover the case when the whole $\mb R^n$ is considered. We believe that, for the first time, this article addresses the non existence result for the inequality \eqref{modeleq}   covering all the dimensions $n\geq 1$ and all the cases $n\leq 2s$ and $n>2s$ in exterior domains of $\mb R^n$ as well as, in the whole of $\mb R^n$, where $n\geq1$. Furthermore, as per our best knowledge,  in the  lower dimensional case $n=1$ with $s\geq \frac 12,$ we have not come across such nonexistence result for \eqref{fq}  involving polynomial nonlinearity even in the whole domain $\mb R$, which is addressed in this article.\par
	One of the  main difficulties  in this article is to  construct suitable functions by implementing appropriate cuts on the fundamental solutions of the fractional Laplace equations which are necessary due to the nonlocal framework we consider here.  Whereas for proving the local version of our results,  in \cite{BC1} the authors directly used the fundamental solutions of Laplace equations.  Though the technique of cutting the fundamental solutions of the fractional Laplace equations was first introduced in \cite{FQ1} for the inequality  like \eqref{fq} in the whole domain  $\mb R^n$, we can not directly adapt that concept in our case  as our problem \eqref{modeleq} is defined  in an exterior domain  of $\mb R^n,n\geq 1$, and the nonlinearity $f$ is more general than that in \cite{FQ1}. Therefore, we face some further challenges to prove our results, in particular,  many suitable cuts on fundamental solutions  are used due to the nonlocal  nature of the operator. Moreover,  we need to carry out some more delicate and different careful analysis compared to the techniques used in \cite{AS1}  to comply with  the nonlocal setting in our problem. \par
	It is important to point out that since, our technique relies solely on properties associated with the maximum principle and the comparison principle for nonlocal operators, our method is applicable to inequalities involving nonlocal operators in both divergence and nondivergence forms, interpreted in classical, weak Sobolev, or viscosity senses, as appropriate.\\
	
	The article is organized as follows: In Section \ref{sec2}, we recall the definitions of  weak and viscosity solutions, and review some versions of the "quantitative" strong maximum principles in the fractional framework. Then, in Section \ref{sec3}, we  state and prove  the first main theorem (Theorem \ref{thm2}) of this article, viz., Liouville type theorem for the inequality \eqref{modeleq}  in  the case $2s\geq n$. Finally, in Section \ref{sectwo}, we address the remaining two main theorems (Theorem \ref{slthm} and Theorem \ref{thm3}), viz., Liouville type  theorems for the inequality \eqref{modeleq} in the case $2s<n$.
	\\\\
	\noi{\bf Notations:}
	\begin{enumerate}
		\item Throughout the paper, $C$ and $c$ denote generic positive constants which may vary from line to line.
		\item $B_r(x)$ denotes ball in $\mb R^n,n\geq1,$ of radius $r$ centered at $x$.
		\item $B_r$ denotes   ball in $\mb R^n,n\geq1,$ of radius $r$ centered at $0$.
		\item $\bar A$ denotes the closure of the set $A$.
		\item $|A|$ denotes the Lebesgue measure of the set $A\subset \mb R^n,n\geq1$.
		\item $C_c^\infty(A)$ denotes the set of all infinitely many times differentiable functions  with compact supports in $A$.
		\item $\chi_A$ denotes the characteristic function, that is $\chi_A=1$ in $A$ and $\chi_A=0$ in $\mb R^n\setminus A,\ n\geq1$.
	\end{enumerate}
	\section{Preliminary tools}\label{sec2}
	Before getting in the proofs of the main theorems in this article, we recall some preliminaries.  For $s\in(0,1),$ let us define the fractional Sobolev space
	\[H^s(\mb R^n)=\left\{ u \in L^2(\mb R^n)\,\,\, | \int_{\mb R^n}\int_{\mb R^n}\frac{|u(x)-u(y)|^2}{|x-y|^{n+2s}} dxdy<+\infty \right\} \] endowed with the norm 
	\[\|u\|_{H^s(\mb R^n)}=\| u \|_ {L^2(\mb R^n)}+ \left( \int_{\mb R^n}\int_{\mb R^n}\frac{|u(x)-u(y)|^2}{|x-y|^{n+2s}} dxdy \right)^{\frac 12}. \]  For any open bounded set $\Om\subset\mb R^n$, we also define the subspace of $H^s(\mb R^n)$
	\[H_0^s(\Om)=\left\{ u \in H^s(\mb R^n)\,\,\, | u=0 \text{ in } \mb R^n\setminus \Om\right\}, \] which is equipped with the norm \[\|u\|_{H_0^s(\Om)}= \left( \int_{\mb R^n}\int_{\mb R^n}\frac{|u(x)-u(y)|^2}{|x-y|^{n+2s}} dxdy \right)^{\frac 12}. \] Note that both $H^s(\mb R^n)$ and $H_0^s(\Om)$ are Banach spaces and $H^s(\mb R^n)=\overline{C_c^\infty(\Om)}^{\|\cdot\|_{H^s(\mb R^n)}};$ $H_0^s(\Om)=\overline{C_c^\infty(\Om)}^{\|\cdot\|_{H_0^s(\Om)}}$. Now let us define the notion of weak solution as following. One may refer to \cite{hit,SV1} and references there in for more details.
	\begin{Definition}[Weak Solutions] 
		Let $\Om$ be an open bounded set in $\mb R^n$.
		A function $u \in H_0^s(\Om)$  is said to be a weak subsolution (or weak supersolution) to $(-\De)^s u = f (u,x) $ in $\Om$; $u=0$ in $\mb R^n\setminus\Om$, 	and we write $\mc (-\De)^s u \leq (\geq) f (u,x) $
		in $\Om$, if  $u\leq (\geq) 0$ in $\mb R^n\setminus \Om$ and for any test function $w\in H_0^s(\Om), w\geq 0,$ 	it holds that
		\[
		\int_{\mb R^n}\int_{\mb R^n}\frac{(u(x)-u(y))(w(x)-w(y))}{|x-y|^{n+2s}} dxdy	
		\leq(\geq) \int_{\mb R^n} f(u(x),x) w(x)dx.\]  	A weak solution is a function $u$ which is both a weak subsolution and a weak supersolution, that is if for  any test function $w\in H_0^s(\Om),$
		\begin{align}\label{wf}
			\int_{\mb R^n}\int_{\mb R^n}\frac{(u(x)-u(y))(w(x)-w(y))}{|x-y|^{n+2s}} dxdy
			=\int_\Om f(u(x),x)w(x)dx.\end{align}
		When $\Om$ is an unbounded open set, then $u$ is a subsolution (or supersolution) to $(-\De)^s u = f (u,x) $ in $\Om$,  if for all bounded open set $\Om^\prime\subset \Om$, $u$ is a  subsolution (or supersolution) to $(-\De)^s u = f (u,x) $ in $\Om^\prime$.
	\end{Definition}
	
	\noi	Next, we define the notion of viscosity solution. For further details, see \cite{BI, CS1, SV1} and the references therein.
	\begin{Definition}[Viscosity Solutions] Let $\Om$ be an open set in $\mb R^n$.
		A function $u : \mb R^n \to \mb R,$ upper (lower)
		semi continuous in $\overline\Om$ and $u\in L^1(\mb R^n,\omega)$ with $\omega=1/(1+|y|^{n+2s})$, is said to be a viscosity  subsolution (or viscosity supersolution) to $(-\De)^s u = f,$
		and we write $\mc (-\De)^s u\leq (\geq) f $ in $\Om$, if every time all the following
		happen\begin{itemize}
			\item	 $x_0$ is any point in $\Om$.
			\item  $N$ is a neighborhood of $x_0$ in $\Om$.
			\item $\varphi$ is some $C^2$ function in $N$.
			\item $ \varphi(x_0) = u(x_0).$
			\item $\varphi (x) < (>)u(x),$  for every $ x \in N\setminus\{x_0\}.$
		\end{itemize}
		Then if we choose as test function 
		$w (x):= 
		\begin{cases}\varphi (x), \text{ for   } x\in N\\
			u (x),\text{ for } x\in \mb R^n\setminus N,
		\end{cases}$\\
		we have $(-\De)^s w(x_0)   \leq (\geq) f (x_0) .$ 
		A viscosity solution is a function $u$ which is both a viscosity subsolution and a viscosity  supersolution.
	\end{Definition}
	We now state a version of the comparison principle for the fractional Laplacian, an important tool in the analysis of nonlocal partial differential equations, which we rigorously use in our proofs. A detailed exposition of this topic can be found in \cite{CS1, CS3, Ro, RS} and the references therein.
	\begin{Theorem}[Comparison Principle ]\label{cppp}
		Let $\Omega \subset \mathbb{R}^N$ be a bounded open set and let $s \in (0,1)$. Let $u_1$ and $u_2$ be weak (viscosity) solutions to
		\[\left\{ \begin{array}{rcll}
			(-\Delta)^s u_1 &=&f_1&\textrm{in }\Omega \\
			u_1&=&g_1&\textrm{in }\R^n\backslash\Omega\end{array}\right.
		\qquad \textrm{and}\qquad
		\left\{ \begin{array}{rcll}
			(-\Delta)^s u_2 &=&f_2&\textrm{in }\Omega \\
			u_2&=&g_2&\textrm{in }\R^n\backslash\Omega.\end{array}\right.\]
		Assume that $f_1\leq f_2$ and $g_1\leq g_2$.
		Then, $u_1\leq u_2$.
	\end{Theorem}
	
	The important aspects in the proof of the Liouville type theorems in this article are the  "quantitative" strong maximum principles in the fractional framework. For that, first we need the following key result which can be found in \cite{bm} and for the sake of  completeness, we give a proof of that result. The   local version of the result is established in \cite[Lemma 3.2]{BC1}.
	\begin{Lemma}\label{mp}
		Let $\Om\subset \mb R^n,\, n\geq1,$ be a smooth    bounded  domain.	Suppose that $h\geq 0$ and $h\in L^\infty(\Om)$. { Let $v\in H_0^s(\mb R^n)$} be a weak solution to the following equation:
		\begin{equation}\label{ol}
			\begin{cases}(-\De)^s v &=  h \;\; \; \text{ in } \Om\\
				v &=0  \quad\text{ in }\; \mb R^n\setminus\Om.
			\end{cases}
		\end{equation}
		Then for all $x\in \Om$, \begin{align}\label{*}\frac{v(x)}{\de^s(x)}\geq C_\Om\int_\Om h(y)\de^s(y) dy,
		\end{align} where $C_\Om>0$ is a positive constant depending only on $s,n$, $\Om$ and $\de(y):=dist (y, \mb R^n\setminus\Om),\, y\in \Om, $ is the distance function.
	\end{Lemma}
	\begin{proof}
		{\bf 	Step 1.} For any compact set $K\subset \Om$, we first show that 
		\begin{align}\label{qm1}
			v(x)\geq C\int_\Om h(y) \de^s(y) dy, \;\text{ for all }\; x\in K,
		\end{align} where $C$ is a positive constant depending only on $s, n,  K$ and $\Om$. Let us set
		\[\alpha:=\frac{dist (K, \mb R^n\setminus\Om)}{4}\] and since $K$ is compact, we can choose $m$ balls of radius $\al$ centered at $x_1, x_2,\cdots, x_m \in K$ such that 
		\[K\subset B_\al(x_1)\cup B_\al(x_2)\cup\cdots\cup B_\al(x_m)\subset \Om.\] Let $\zeta_1,\zeta_2, \cdots,\zeta_m\in H_0^s(\Om)$ be the weak solutions of 
		\begin{equation}\label{oll}
			\begin{cases}(-\De)^s \zeta_i  &=  \mc X_{B_{\al}(x_i)} \;\; \; \quad in \;\;  \Om\\
				\zeta_i &=0  \;\;\; \; \;\;\;\;\;\qquad in \;\; \mb R^n\setminus\Om,
			\end{cases}
		\end{equation} where $i=1,2,\cdots, m.$
		Then, by the fractional Hopf's lemma {(see \cite{ AP1, Ro})}, there exists a constant $C:=C(s,n,\Om)>0$ such that, for all $x\in \Om, $ 
		\begin{align}\label{qm2}
			\zeta_i(x)\geq C\de^s(x), \; \text { for all } 1\leq i\leq m.
		\end{align}Fix $z\in K$ and choose a ball $B_\al(x_i)$ such that $z\in B_\al(x_i) $. 
		Then, $B_\al(x_i)\subset B_{2\al}(z)\subset B_{4\al}(z)\subset \Om$. Since $(-\De)^sv(x)=h(x)\geq 0$ in $\Om$,  by the maximum principle (see \cite{SV1, Ro}), $ v(x)\geq 0$ in $\mb R^n.$ Therefore,  $(-\De)^sv(x)\geq 0$ in $ B_{4\al}(z)$. For any $x\in \mb R^n,$ define the function
		\[w_\al(x):=v\left(2\al x+z\right).\] Clearly, $(-\De)^sw_\al \geq 0$ in $ B_2$  and $ w_\al\geq0$ in $\mb R^n$. Now recalling    {\cite[Proposition 3.3.4]{bukur}}, using  the weak formulation \eqref{wf} for \eqref{ol} with $u=v,\, w= \zeta_i$ and for  \eqref{oll} with $u=\zeta_i,\,w=v$, and then  plugging \eqref{qm2} in it, we deduce that
		\begin{align*}\label{dya}
			v(z)=w_{\al}(0)&>C\int_{B_1} w_\al(\tau)d\tau\notag\\
			&=C
			\int_{B_1} v(2\al \tau+z)d\tau
			\notag\\
			&	=C\int_{B_{2\al}(x)} v(y)dy\quad\text{(using change of variable $y=2\al \tau+z$)}\notag\\
			&\geq C\int_{B_\al(x_i)}v(y) dy\notag\\
			&=C
			\int_{\mb R^n}\int_{\mb R^n}\frac{(v(x)-v(y))(\zeta_i(x)-\zeta_i(y))}{|x-y|^{n+2s}} dxdy	\notag\\
			&=C\int_\Om h(y)\zeta_i(y)\, dy
			\geq 
			C\int_\Om h(y)\de^s(y)\, dy,
		\end{align*} where $C>0$ is a generic constant depending only on $s,n,\Om, K$ and  which may vary from line to line. Thus, we get \eqref{qm1}.\\
		{\bf Step  2.} Fix a smooth compact set $K\subset \Om$. By \eqref{qm1}, it follows that
		\begin{equation}\label{ck}	
			v\geq C_K\int_{\Om} h(y)\de^s (y)dy \qquad \text{ in } K,	\end{equation}
		
		where $C_K>0$ is a positive  constant which depends  only on $s,n,\Om, K.$ 	So, it is sufficient to prove \eqref{*},  for all $x\in \Om\setminus K.$
		Now, using the last relation, we can rewrite \eqref{ol} as
		\begin{equation}\label{vr}
			\begin{cases}(-\De)^s V_K&=  0 \; \;\; \quad\text{ in } \Om\setminus K\\
				V_K &=0  \;\; \;\; \quad\text{ in }\; \mb R^n\setminus\Om\\
				V_K &\geq 1   \;\; \;\; \quad\text{ in }\; K,
			\end{cases}
		\end{equation}   where $V_K(x):=\ds \frac {v(x)}{C_K\int_\Om h(y)\de^s (y)dy}$ and $C_K$ is as in \eqref{ck}.
		Let $w$ be a solution of 
		\begin{equation}\label{wr}
			\begin{cases}(-\De)^s w &=  0 \; \;\; \quad\text{ in } \Om\setminus K\\
				w &=0  \;\; \;\; \quad\text{ in }\; \mb R^n\setminus\Om\\
				w &=1   \;\; \;\; \quad\text{ in }\; K.
			\end{cases}
		\end{equation}
		Again, using the fractional Hopf's lemma \cite{AP1, Ro} to \eqref{wr},  we obtain
		\begin{align}\label{qm3}
			w(x)\geq C\de^s(x), \; \text{for all } x\in \Om\setminus K,
		\end{align} where $C>0$ is a positive  constant  depending  only on $s,n,\Om, K.$  
		Now applying  the comparison principle for the fractional Laplacian (Theorem \ref{cppp}) to \eqref{vr} and \eqref{wr}, and then using \eqref{qm3}, we deduce
		\[v(x)\geq C_K\left(\int_\Om h(y)\de^s(y) dy\right) w(x)\geq C _{\Om \setminus K}\left(\int_\Om h(y)\de^s(y) dy\right) \de^s(x),  \; \text{ for all } x\in\Om\setminus K,\] where and $C_K$ is as in \eqref{ck} and $C_{\Om \setminus K}>0$ is a positive  constant which depends  only on $s,n,\Om, K.$ This and \eqref{ck}  with $C_\Om:=\min\{C_K,C_{\Om \setminus K}\}$ complete the proof of the lemma. 
	\end{proof}	\hfill{\QED}
	Using  Lemma \ref{mp}, we can derive the following version of  nonlocal "quantitative" strong maximum principle  for  the weak sub-supersolution, which is useful to prove our main theorems in this article. The similar result for the local Laplace  operator was stated in \cite[Lemma 2.2]{AS1}.
	\begin{Proposition}\label{kslap}
		Assume $h \in L^\infty\left (B_3\setminus B_{\frac 12}\right)$ is non-negative, 
		{and $u\geq 0$} in $\mb R^n, \,n\geq 1,$ and it satisfies
		\begin{equation}\label{kkv}
			(-\Delta)^s u \geq h  \quad \mbox{in} \ B_3\setminus B_{\frac 12}
		\end{equation} in the  "weak" sense.
		Then, there exists a constant $\bar c>0$ depending only on $s,n$ such that  for each $A\subset B_2\setminus B_1$
		$$
		\inf_{B_2\setminus B_1}u \geq \bar c |A|\inf_A h.
		$$
	\end{Proposition}
	\begin{proof}
		Consider $\Om=B_3\setminus B_{\frac 12}$ in Lemma \ref{mp}.
		Now   applying  the comparison principle  to \eqref{ol} and \eqref{kkv}, we find that $u\geq v$ in $\mb R^n$. Therefore, using \eqref{*}, the result follows.
	\end{proof}\hfill{\QED}
	Another form of the "quantitative" strong maximum principle for the fractional Laplacian is proved below  for the  weak sub-supersolution. The similar  result in the local framework was established in    \cite[Theorem 3.3]{AS1}.
	\begin{Proposition} \label{qsmp}
		Let $K$ and $A$ be two compact subsets of a smooth bounded domain $\Omega\subset\R^n$, with $|A| > 0$. Suppose that { $v\geq 0$  in $\mb R^n, \, n\geq 1,$} and it satisfies
		\begin{equation*}
			(-\Delta)^s v\geq \chi_A \quad \mbox{in} \ \Omega
		\end{equation*} in the "weak"  sense,
		where $\chi_A$ denotes the characteristic function of $A$.
		\begin{enumerate}\item[(i)] Then there exists a constant $c_0=c_0(s,|A|,\Omega,K)>0$ such that
			\begin{equation*}
				v \geq  c_0\quad\mbox{on }\; K.
			\end{equation*}
			\item[(ii)] Suppose in addition that $v\ge \Phi^*\ge 0$ in $\mb R^n\setminus \Omega$, where $\Phi^*=\Phi$ or $=\tilde\Phi$ as per the positive sign of them (where $\Phi,\, \tilde \Phi$ are defined in \eqref{fs}) and $0\not\in \Omega$. Then there exists a constant $c_0=c_0(s,|A|,\Omega,K)>0$ such that
			\begin{equation*}
				v \geq \Phi^*+ c_0\quad\mbox{on }\; K.
			\end{equation*}
		\end{enumerate}
	\end{Proposition}
	\begin{proof}
		We prove this by contradiction. So,	we suppose that either of  $(i)$ or $(ii)$ is false. 
		Therefore, there exist a sequence of compact sets $\{A_j\}$ $(A_j\subseteq \Omega)$ such that  $\ds\inf_j |A_j|>0$, and a sequence of positive functions $\{v_j\}$ in $H^s(\mb R^n)$ with
		
		\begin{align}\label{jk}(-\Delta)^s v_j \ge \chi_{A_j} \;\;\text { in } \Omega.\end{align} Moreover, there exists a sequence of  points $\{x_j\}$ in $K$ such that
		\begin{equation} \label{pass0}
			\mbox{either} \quad  v_j(x_j) \to 0 \qquad \mbox{or} \quad v_j(x_j)-\Phi^*(x_j)\to 0 \quad \mbox{as} \ j \to +\infty.
		\end{equation}
		Let $\tilde v_j\in H_0^s(\Om)$ solve the Dirichlet problem \begin{equation}\label{dpok}
			\begin{cases}	(-\Delta)^s \tilde v_j &= \chi_{A_j}\;\mbox{ in }\;\Omega\\
				\tilde v_j&=0 \quad(\mbox{or }\; \tilde v_j = \Phi^*) \;\mbox{ in }\;\mb R^n\setminus \Omega.\end{cases}
		\end{equation}
		Then applying the comparison principle {(Theorem \ref{cppp})} to \eqref{jk} and \eqref{dpok}, we get $$ \tilde v_j \leq v_j\text {  in } \mb R^n,$$ and thus, we can replace $v_j$ by $\tilde v_j$ in \eqref{pass0}.  Now  since, $C^\infty_c(\Omega)$ is dense in $H_0^s(\Om)$, for all  $\varphi \in C^\infty_c(\Omega)$, using the weak formula \eqref{wf} for \eqref{dpok} with $u=\tilde v_j$ and $w=\varphi$, we have
		\begin{equation}\label{pass1}
			\int_{\mb R^n}\int_{\mb R^n}\frac{(\tilde v_j(x)-\tilde v_j(y))(\varphi(x)-\varphi(y))}{{|x-y|^{n+2s}}}\, dxdy = \int_{A_j} \varphi (x)\,dx.
		\end{equation}
		Recalling the regularity estimates for the fractional Laplace equations (see \cite{CS1, CS3, RS}), we find that $\{\tilde v_j\}$ is bounded in $C^{0,\alpha}(\overline\Omega)$ for some $\alpha:=\al(s,\Om)\in (0,1)$. Since $C^{0,\al}(\overline \Om)$ is compactly embedded into $C( \overline\Om)$, we may extract a subsequence of $\{\tilde v_j\}$ which converges to some $v_0$ in $C(\overline{\Omega})$ as $j\to+\infty$. Now passing to the limit $j\to+\infty$ in \EQ{pass1}, by standard variational arguments,  it yields that  $(-\Delta)^s v_0\ge 0$ in $\Omega$, as well as $v_0\ge 0$ in $ \mb R^n\setminus\Omega$ (or $v_0\ge \Phi^*$ in $\mb R^n\setminus\Omega$).
		Therefore, for the case $(i)$, applying the strong maximum principle (see \cite{BI,CS1,SV1}), we get that  either $v_0\equiv 0$ or $v_0>0$ in $\Omega$.\\\par
		For the case $(ii)$, by applying the strong comparison principle {(see \cite{BI,CS1,SV1})}, we infer that either $v_0\equiv\Phi^*$ or $v_0>\Phi^*$ in $\Omega$. 
		Next, 	by passing to limit $j\to+\infty$ in \EQ{pass0}, we find that $v_0\equiv0$ in $\Omega$, or  $v_0\equiv\Phi^*$ in $\Omega$ (for  the case $(ii)$). Hence, in both the cases $(i)$ and $(ii)$, it follows that  $(-\De)^sv_0=0$ in $\Om$.  Therefore, for every $\varphi \in C^\infty_c(\Omega)$,  passing to the limit in \re{pass1}, we obtain
		\begin{equation*}
			\lim_{j\to+\infty} \int_{A_j} \varphi(x) \,dx \to 0 \quad \mbox{as}\  j \to +\infty.
		\end{equation*}
		By taking $\varphi \ge 1$ except on a very small subset of $\Omega$ and recalling  that  $\ds\inf_j |A_j| > 0$, from the last limit,  we arrive at  a contradiction. This concludes the proof.
	\end{proof}\hfill{\QED}
	\begin{Remark}
		In the previous two propositions,  we proved  the versions of the "quantitative" strong maximum principle (Proposition \ref{kslap}, Proposition \ref{qsmp}) in  the case of weak supersolutions.
		To obtain those results for viscosity supersolutions, we observe that the solution $v$ in Lemma \ref{mp} is also a viscosity solution by \cite[Theorem 1]{SV1} so the Proposition \ref{kslap}  follows by the  comparison principle for viscosity solutions (see Theorem \ref{cppp}). 
		Moreover, to prove  Proposition \ref{qsmp} in the viscosity sense, we first solve \eqref{ol} by taking $h:=\chi_A$ and call the solution as $z$ and again by  \cite[Theorem 1]{SV1}, $z$ is a viscosity solution for \eqref{ol}. Now  by the  comparison principle for viscosity solutions $v\geq z$ in $\mb R^n$. Thus, the result follows.
		
	\end{Remark}
	In the proofs of our main theorems, we make use of the following crucial result, whose proof follows from \cite[Lemma 10.1]{CS1} for the viscosity supersolution (it also can be obtained by the weak Harnack inequality) and from \cite[Theorem 5.8]{kass} for the weak supersolution. For the local version of this result, see \cite[Lemma 2.4]{AS1}.
	\begin{Lemma}\label{vwhlap}
		For every $0 < \nu < 1$, there exists a constant $\bar C = \bar C(s,n,\nu)>1$ such that for any non-negative  function $u$ with $(-\De)^su\geq 0$ in $B_3 \setminus \bar B_{1/2}$ and  for any $x_0 \in B_2 \setminus B_1$, we have
		\begin{equation*}
			\left| \left\{ u \leq \bar C u(x_0) \right\} \cap (B_2\setminus B_1) \right| \geq \nu \left|B_2 \setminus B_1\right|.
		\end{equation*}
	\end{Lemma}	
	\begin{Remark}
		In the next two sections, we will deliberately leave the statements of the theorems and the proofs of the related lemmas and theorems vague as to the notion of solution since all the results hold irrespective of the choice of classical, weak, viscosity solutions.
	\end{Remark}


	\section{Liouville  type theorem : case $2s\geq n$}\label{sec3}
	In this section, for the case $2s\geq n,$ that is, for $n=1$ and $s\in \big[\frac 12, 1\big)$, we show the non existence result for the inequality 
	\begin{equation} \label{slinxs0}
		(	-\Delta)^s u \ge f(u,x)
	\end{equation}
	in  an exterior domain $\mb R\setminus B_{r_0}$ for a $r_0>0.$ We consider the following hypotheses on  the nonlinearity $f:=f(t,x)$:
	{	\begin{itemize}
			\item[$(f1')$] $f: (0, +\infty) \times (\R \setminus B_{r_0}) \to (0,+\infty)$ is continuous.
			
			\item[$(f2')$] $|x|^{2s}f(t,x) \to +\infty$ as $|x|\to +\infty$ locally uniformly in $t\in (0, +\infty)$.
			\item[$(f3')$] 
			Let $\sigma^*:=-1+2s \geq 0$. Then there exists a constant $\underline \mu > 0$ such that  if we define
			\begin{equation*}
				\widetilde\Psi_k(x) := |x|^{2s} \inf_{\underline\mu\le t\le k\widetilde\Phi(x)}  \frac{f(t,x)}{t }\quad\mbox{and}\quad \widetilde h(k):=  \liminf_{|x|\to +\infty} \widetilde\Psi_k(x),
			\end{equation*} where $ \tilde \Phi=-\Phi$ given  in \eqref{fs}, 
			then $0 < \widetilde h(k) \leq +\infty,$ for each $k>0$, and
			\begin{equation*}
				\lim_{k\to 0} \widetilde h(k) =+\infty.
			\end{equation*}
	\end{itemize}}
	\noi For the reader's convenience, some of the models of such nonlinearities $f$ are described below.
	\begin{Remark}\label{grmk} \begin{enumerate}
			\item Consider the nonlinearity $f$ of the form \eqref{gbz}. Then clearly, $(f1^\prime)$ and $(f2^\prime)$  imply that $g:(0,+\infty)\to(0,+\infty)$ is continuous. Together with this, a sufficient condition for $(f3^\prime)$ is:
			\begin{align}\label{z2} \text {if $ s>\frac 12$, {then}  $\ds \liminf_{t\to +\infty} t^{-\widetilde \al^*}g(t) > 0$,  where  $ \widetilde\al^*:= 1+ \frac{2s-\gamma}{-\sigma^*};$}\end{align}
			\begin{align} \label{z1}\text {if $ s=\frac 12$, then    $\ds \liminf_{t\to +\infty} e^{bt} g(t) > 0, $ {for every} $b>0.$}\end{align}
			Indeed, first
			observe  that for any $\underline\mu>0$,  from \eqref{z2}, we get
			\begin{equation*}
				\eta(\underline\mu):= \inf_{\underline\mu \leq t} t^{-\widetilde\al^*} g(t) > 0.
			\end{equation*}
			Thus, with $\widetilde \Psi_k(x)$ as in $(f3^\prime),$  for each such $\underline\mu > 0$ and for all sufficiently large $|x|$, we deduce
			\begin{equation*}
				\widetilde\Psi_k(x) = |x|^{2s} \inf_{\underline\mu\le t\le k\widetilde\Phi(x)} \frac {f(t,x)}{t } \geq \eta(\underline\mu) \inf_{\underline\mu\le t\le k\widetilde\Phi(x)} |x|^{2s-\gamma} t^{-1+{\tilde\al}^*}.
			\end{equation*}
			Since $\tilde \Phi(x)\to 0$ as $|x|\to +\infty,$ as $\sigma^*=-1+2s>0$, the infimum in the last relation  is attained at $t = k\tilde\Phi(x)$.  This fact, together with the relation  $-1+\widetilde \al^* = \frac{2s-\gamma}  {-\sigma^*} <0$, yields that 
			\begin{equation*}
				\tilde\Psi_k(x) \geq c \eta(\underline\mu) k^{(-1+\tilde\al^*)} |x|^{2s-\gamma-\sigma^*(-1 + \widetilde\al^*)} = \eta(\underline\mu) k^{(-1+\widetilde\al^*)} \rightarrow +\infty \quad \mbox{as} \ k \to 0.
			\end{equation*}
			Thus,  $(f3^\prime)$ holds.  In a similar  way,  we can argue for the case $s=\frac 12$ with the condition \eqref{z1}. 
			\item Consider a function  $f$ (in the spirit of \cite{chen2}) satisfying $(f1^\prime), (f2^\prime)$ and $\frac{f(t,x)}{t}$ is decreasing in $t>0$ for all $x\in \mb R\setminus B_{r_0}.$
			Moreover, for each $k>0$,
			\[|x|f(k|x|^{-1+2s},x)>z_1(k), \text{ when $\frac12<s<1$ and } \frac {|x|f(k\log|x|,x)}{\log|x|}>z_1(k),  \text{ when $s=\frac12$,} \] for some sublinear function $z_1$ with  $\ds\lim_{k \to 0} \frac{z_1(k)}{k}=+\infty.$ One of the standard examples for $z_1$ is  $z_1(k)=a_1 k^{\delta}$, $\delta<1$, with some  constant $a_1>0$. The above conditions imply $(f3^\prime).$ 
		\end{enumerate}
	\end{Remark}
	\noi Now we state the first main theorem of this article below:
	\begin{Theorem}\label{thm2}
		Let $s\in \big[\frac 12,1\big)$  and the nonlinearity $f$ satisfy the hypotheses $(f1')$-$(f3')$. Then, the  inequality \re{slinxs0} has no continuous positive  solution in any exterior domain of $\mb R$.
	\end{Theorem} 
	Before proceeding to prove Theorem \ref{thm2}, we first prove the series of results as following. 
	\begin{Lemma} \label{mbounds_1}
		Let  $s\in\left(\frac 12,1\right)$.  Assume that $u$ is  a non-negative function in $\mb R$ such that, $u$ is  continuous and $u>0$  in $\R \setminus B_{r_0}$, for some $r_0 > 1$, satisfying   
		\begin{equation}\label{ea1}
			(-\Delta)^su \geq 0 \quad \mbox{in} \  \R\setminus B_{r_0}.
		\end{equation}
		Then, there exist two constants $	\bar c_{\min},	\bar C_{\max}>0$ depending only on $s$, $u$ and $r_0$ such that 
		\begin{equation} \label{mbdsb}
			{	\bar c_{\min}}\leq \ds\inf_{B_{2r} \setminus B_r}u\leq {\bar C_{\max}} r^{-1+2s}.
		\end{equation}
	\end{Lemma}
	\begin{proof} For any $r>r_0>1,$ define
		$$\tilde m(r):=\ds\inf_{B_r} u.$$ By \eqref{fs},  $\Phi(x)=-|x|^{-1+2s}\to-\infty$ as $|x|\to+\infty.$ Now for any $r>r_0$ and $\e>0$, we can choose sufficiently large $\tilde R:=\tilde R(\e,r_0,r)>r$  such that 
		\begin{align}\label{r1}
			\tilde m(r)+\e\Phi(x)\leq u\;\;\text{ in } \mb R\setminus B_{\tilde R}.
		\end{align}
		It easily follows that
		\begin{align}\label{r2}
			\tilde m(r)+\e\Phi\leq u  \;\;\text{ in }\ \overline B_r.
		\end{align} Also, we have 
		\begin{align}\label{r3}
			(-\De)^s(\tilde m(r)+\e\Phi)=0\leq (-\De)^su\;\;\text{ in } B_{\tilde R} \setminus B_r.
		\end{align} So, by applying the comparison principle to \eqref{r1}, \eqref{r2} and \eqref{r3}, we deduce
		\[ \tilde m(r)+\e\Phi (x)\leq u(x), \;\text{ for all } x\in \mb R.\] Now taking $\e\to 0$ in the last relation, we get
		\[\tilde m(r)\leq u(x),\; \text{ for all  } x\in\mb R.\] So, for any fixed $r_1>r_0,$ we  have $0<\tilde m(r_1)\leq u(x),$  for all  $x\in \mb R$. Thus,  for any $r>r_0,$ taking infimum over  ${B_{2r}\setminus B_r}$ in the last relation, there exists some constant $\bar c_{\min}$, independent of $r$, such that it yields the first inequality in \eqref{mbdsb}.\\
		To get the second inequality in \eqref{mbdsb}, for any $r>r_0,$ we define the following two functions \begin{equation}\label{cf20}
			\tilde v (x)  =
			\left\{
			\begin{array}{ll}
				r_0^{-\sigma^*}|x|^{\sigma^*}-1 & \mbox{if } |x| < 2r\\
				0 & \mbox{if } |x| \geq 2r,
			\end{array}
			\right.
		\end{equation} 
		\begin{equation}\label{cf2r} \tilde g (x)  =\left\{
			\begin{array}{ll}
				0  & \mbox{if } |x|\leq \frac{3r}{2}\\
				\tilde C_{\tilde g} |x|^{\sigma^*}& \mbox{if } \frac{3r}{2}<|x| < 2r\\
				0 & \mbox{if } |x| \geq 2r,
			\end{array}
			\right.
		\end{equation}  where $\sigma^* :=-1+2s>0$ and $\tilde C_{\tilde g}$ is a positive constant, independent of $r$, to be determined later.
		Now for any $r>r_0$ and for all $x\in B_r\setminus B_{r_0}$, using that $|x|^{\sigma^*}$ is a fundamental solution for $(-\De)^s,$ we deduce the following estimate:
		\begin{align}\label{ca3d}
			&(-\Delta)^s(\tilde v (x))
			\notag\\
			&=C_{1,s}\left(P.V. \int_{B_{2r}(x)}\frac{\tilde v(x)-\tilde v(x-y)}{|y|^{1+2s}} dy+ \int_{\mb R\setminus B_{2r}(x)}\frac{\tilde v(x)-\tilde v(x-y)}{|y|^{1+2s}} dy\right)\notag\\
			&= C_{1,s} \left(P.V.\int_{B_{2r}(x)}\frac{(r_0^{-\sigma^*}|x|^{\sigma^*}-1)-(r_0^{-\sigma^*}|x-y|^{\sigma^*}-1)}{|y|^{1+2s}} dy+ \int_{\mb R\setminus B_{2r}(x)}\frac{r_0^{-\sigma^*}|x|^{\sigma^*}-1}{|y|^{1+2s}} dy\right)\notag\\&\qquad\pm  \int_{\mb R\setminus B_{2r}(x)}\frac{(r_0^{-\sigma^*}|x|^{\sigma^*}-1)-(r_0^{-\sigma^*}|x-y|^{\sigma^*}-1)}{|y|^{1+2s}} dy\notag\\
			&=(-\De)^s(r_0^{-\sigma^*}|x|^{\sigma^*}-1)+C_{1,s}\int_{\mb R\setminus B_{2r}(x)}\frac{r_0^{-\sigma^*}|x-y|^{\sigma^*}-1}{|y|^{1+2s}} dy \notag\\
			&\leq 0+C_{1,s}\int_{\mb R\setminus B_{2r}(x)}\frac{r_0^{-\sigma^*}|x-y|^{\sigma^*}-1}{(\frac 12|x-y|)^{1+2s}} dy \qquad\quad \text {( since for $y\in \mb R\setminus B_{2r}(x), \, |y|\geq \frac 12 |x-y|$)} \notag\\
			&\leq C(s) \int_{2r}^{+\infty}\frac{(r_0^{-\sigma^*}t^{\sigma^*}-1)}{t^{1+2s}} dt \qquad\qquad \qquad\text{  (taking $|x-y|=t$)}\notag\\
			&=C(s) \left(-r_0^{-\sigma^*}\frac{(2r)^{\sigma^*-2s}}{\sigma^*-2s}+\frac {(2r)^{-2s}}{2s} \right)\notag\\
			&=C_1(s,r_0)\left(\frac 1r+\frac{1}{r^{2s}}\right)
			\notag\\ &
			\leq2 C_1(s,r_0) \frac{1}{r},
		\end{align} where $C(s), C_1(s,r_0)$ are two positive constants. 
		Similarly, for any $r>r_0$ and for all $x\in B_r\setminus B_{r_0},$ we estimate the following:			
		\begin{align}\label{ca3pr}
			&(-\Delta)^s\tilde g(x)\notag\\
			&= C_{1,s}\left(\int_{{B_{2r}(x)\setminus B_{\frac{3r}{2}(x)}}}\frac{\tilde g(x)-\tilde g(x-y)}{|y|^{1+2s}} dy+P.V. \int_{\mb R\setminus\left({B_{2r}(x)\setminus B_{\frac{3r}{2}(x)}}\right)}\frac{\tilde g(x)-\tilde g(x-y)}{|y|^{1+2s}} dy\right)\notag\\
			&= C_{1,s}\int_{B_{2r}(x)\setminus B_{\frac{3r}{2}(x)}}\frac{ -\tilde C_{\tilde g}|x-y|^{\sigma^*}}{|y|^{1+2s}} dy+0\notag\\
			&\leq C_{1,s} \int_{B_{2r}(x)\setminus B_{\frac{3r}{2}}(x)}\frac{ -\tilde C_{\tilde g}|x-y|^{\sigma^*}}{|3r|^{1+2s}} dy\qquad\quad \text{(since $\, |y| \leq |x-y|+|x|\leq 2r+r=3r$)} \notag\\
			&\leq-\tilde C_{\tilde g} \tilde C(s)\, \frac{1}{r^{1+2s}}\int_{\frac{3r}{2}}^{2r} t^{\sigma^*} dt\qquad\quad \text{(taking  $|x-y|=t$)}\notag\\
			&=-\tilde C_{\tilde g} C_2(s,r_0) \frac 1r, 
		\end{align} where $\tilde C(s), C_2(s,r_0)$ are two positive constants. 
		Now we choose the constant  $\tilde C_{\tilde g}>0$ such that $\tilde C_{\tilde g}\,C_2(s,r_0)>2C_1(s,r_0)$, that is, from \eqref{ca3d} and \eqref{ca3pr}, it follows that \begin{align}\label{lvc}(-\De)^s(\tilde v+\tilde g)(x)<0, \; \text{ for all } x\in B_r\setminus B_{r_0}.\end{align} For any $r>r_0$, using \eqref{cf20} and \eqref{cf2r}, define the  function 
		$$ \hat h(x):={\tilde v(x)+\tilde g(x)}
		$$ and	set the quantity $$\tilde \rho(r):=\inf_{B_{2r}\setminus B_{r}}\frac{u}{\hat h}>0.$$
		Now   combining  \eqref{ea1} and \eqref{lvc}, for any $r>r_0,$ we obtain \begin{align}\label{ghb}(-\De)^s\tilde \rho(r)\hat h(x)<0\leq (-\De)^su(x),\;\text{ for all  }   x\in B_{r}\setminus  B_{r_0}.\end{align}  Moreover,  by the definition of $\hat h$,  for any $r>r_0$,  we get  $\tilde \rho(r)\hat h\leq0\leq u$ in $\overline {B}_{r_0}$ and 
		$\tilde\rho(r)\hat h\leq u$ in $\mb R\setminus  B_{r}$. Hence, applying the  comparison principle to \eqref{ghb}, for any sufficiently large  $r>2r_0$, we get
		\[\tilde \rho(r)\hat h\leq u \quad \text {  in  }   \mb R,\] which yields that  \[ 
		\tilde \rho(r)  \inf_{B_{3r_0}\setminus B_{2r_0}}\hat h\leq \inf_{B_{3r_0}\setminus B_{2r_0}} u. \]
		Hence, the second inequality in \eqref{mbdsb} is achieved. This completes the proof of the lemma.
	\end{proof}\hfill{\QED}
	\begin{Lemma} \label{mb2} 	Let  $s=\frac 12$ and suppose that   $u$ is  a non-negative function in $\mb R$ such that,  $u$ is   continuous  and $u>0$ in $\R \setminus B_{r_0}$, for some $r_0 > 1$,  and $u$ satisfies   
		\begin{equation}\label{ea2}
			(-\Delta)^su \geq 0 \quad \mbox{in} \  \R\setminus B_{r_0}.
		\end{equation}
		Then,  there exist two constants $	\underline c_{\min},\underline C_{\max}>0$ depending only on $s$, $u$ and $r_0$ such that 
		\begin{equation} \label{mbds}
			\underline c_{\min}\leq \ds\inf_{B_{2r} \setminus B_r}u\leq \underline C_{\max} \log r.
		\end{equation}
	\end{Lemma}
	\begin{proof}
		For $r>r_0>1,$ as in Lemma \ref{mbounds_1}, define $$\tilde m(r):=\ds\inf_{B_r} u.$$ Observe that by \eqref{fs}, the fundamental solution of $(-\Delta)^s$ for the case $2s=1$ in $\mb R$ is $\Phi(x)=-\log |x|\to-\infty$ as $|x|\to+\infty.$ For  $r>r_0$, Define 
		\begin{equation*}
			\check w(x)=
			\begin{cases}
				0 &\quad\text{ if } |x|\leq r\\
				-\log |x| &\quad\text{ if  } |x|>r.
			\end{cases}
		\end{equation*} Now for any $r>r_0$ and $\e>0$, we can choose  sufficiently large $\widehat R:=\widehat R(\e,r_0,r)>r$  such that 
		\begin{align}\label{r1.}
			\tilde m(r)+\e\check w\leq u \;\; \text{ in }\ \mb R\setminus B_{\widehat R}.
		\end{align}
		Also, for any $r>r_0,$	it easily follows that
		\begin{align}\label{r2.}
			\tilde m(r)+\e\check w\leq u\;\;\mbox{ in } \ \overline B_{r}.
		\end{align} Next, for any $r>r_0$ and for all $x\in B_{\widehat R}\setminus B_r,$  using that $-\log|x|$ is a fundamental solution for $(-\De)^s$, we derive 
		\begin{align}\label{ca1.00}
			&(-\Delta)^s{(\tilde m(r)+\e\check w(x))}\notag\\&=\e(-\De)^s\check w(x)\notag\\
			&= C_{1,\frac 12}\left(P.V.\; \e\int_{\mb R\setminus B_{r}(x)}\frac{\check w(x)-\check w(x-y)}{|y|^{2}} dy+ \e \int_{ B_{r}(x)}\frac{\check w(x)- \check w(x-y)}{|y|^{2}} dy\right)\notag\\
			&=C_{1,\frac 12}\Bigg(P.V.\; \e \int_{\mb R\setminus B_{r}(x)}\frac{-\log|x|+\log |x-y|}{|y|^{2}} dy+ \e\int_{ B_{r}(x)}\frac{-\log |x|}{|y|^{2}} dy\notag\\
			&\qquad\pm \e\int_{ B_{r}(x)}\frac{-\log |x|+\log |x-y|}{|y|^{2}} dy\Bigg)\notag\\
			&=(-\De)^s(-\log|x|)+C_{1,\frac 12} \e \int_{ B_{r}(x)}\frac{-\log |x-y|}
			{|y|^{2}} dy\notag\\
			&\leq  0+ C_{1,\frac 12}\left(\e \int_{ B_{1}(x)}\frac{-\log |x-y|}
			{|y|^{2}} dy + \e \int_{ B_{r}(x)\setminus B_1(x)}\frac{-\log |x-y|}
			{|y|^{2}} dy\right).\end{align}
		Now we estimate the first integration in \eqref{ca1.00} as following:
		\begin{align}\label{gh1}
			& \e \int_{ B_{1}(x)}\frac{-\log |x-y|}	{|y|^{2}} dy\notag\\
			&\leq  \e \int_{ B_{1}(x)}\frac{-\log |x-y|}
			{|r_0-1|^{2}} dy\qquad\text{(since $|y|\geq|x|-|x-y|\geq r_0-1$)}\notag\\
			&\leq C_3 \e \int_0^1
			\frac{-\log t}{(r_0-1)^{2}} dt \qquad\qquad\text{(taking  $|x-y|=t$)}\notag\\
			&= C_3\frac{\e}{(r_0-1)^{2}}[t-t\log t]_0^1\notag\\
			&=C_3\frac{\e}{(r_0-1)^{2}}\qquad\text{(since $\lim_{t\to0}t\log t=0$ by the L´Hospitals rule)},
		\end{align} where $C_3$ is a positive constant that does not depend on $r$. Similarly, we estimate the second integration in \eqref{ca1.00} as following:
		\begin{align}\label{gh2}
			& \e \int_{ B_{r}(x)\setminus B_1(x)}\frac{-\log |x-y|}	{|y|^{2}} dy\notag\\
			&\leq  \e \int_{ B_{r}(x)\setminus B_1(x)}\frac{-\log |x-y|}
			{|2R|^{2}} dy\qquad\text{(since $|y|\leq|x|+|x-y|\leq 2\widehat R$)}\notag\\
			&\geq C_4\e \int_1^r
			\frac{-\log t}{(2\widehat R)^{2}} dt \qquad\qquad\text{(taking  $|x-y|=t$)}\notag\\
			&=C_4 \frac{\e}{(2\widehat R)^{2}}[t-t\log t]_1^r\notag\\
			&=C_4\frac{\e}{(2\widehat R)^{2}}(r-r\log r-1)
			\notag\\&
			<0,\;\text{for sufficiently large $r>r_0$},
		\end{align} where $C_4$ is a positive constant, independent of $r$.
		So,  combining \eqref{ca1.00}, \eqref{gh1} and \eqref{gh2}, for  sufficiently large $r>r_0,$ we have  $(-\De)^s(\tilde m(r)+\e\check w)<0$ in $B_{\widehat R}\setminus B_{r}.$ This and \eqref{ea2}, for every sufficiently large $r>r_0,$ imply that
		\begin{align}\label{r3.}
			(-\De)^s(\tilde m(r)+\e\check w(x))\leq (-\De)^su(x), \; \text{ for  all } x\in  B_{\widehat R}\setminus B_{r}.
		\end{align} Therefore, for every  sufficiently large $r>r_0$,  applying the  comparison principle to \eqref{r1.}, \eqref{r2.} and \eqref{r3.},  we obtain
		\[ \tilde m(r)+\e\check w(x)\leq u(x),\; \text{ for all } x\in  \mb R.\] Now passing to the limit $\e\to 0$ in the last relation,  we get, for any $r>r_0$, that
		\[ \tilde m(r)\leq u(x), \; \text{ for all  } x\in \mb R.\] Therefore, for any fixed $r_2>r_0>1$ large,  we have  $ 0<\tilde m(r_2)\leq u(x),$  for all $\mb R.$ Hence,   for any sufficiently large $r>r_0,$ taking infimum  over $B_{2r}\setminus B_r$ in  the last relation,  we obtain  the first inequality in \eqref{mbds}.\\
		\noi Next, to prove the second inequality in \eqref{mbds},   for any $r>r_0$, we define the following function   \begin{equation}\label{cfr2}\check v (x)  =
			\left\{
			\begin{array}{ll}
				{\log|x|}& \mbox{if } |x| < 2r\\
				0 & \mbox{if } |x| \geq 2r.
			\end{array}
			\right.
		\end{equation}  In addition to that, for any $r>r_0$, we  define another function $\check g
		$ as \begin{equation}\label{cfr} \check g (x) =\left\{
			\begin{array}{ll}
				0  & \mbox{if } |x| \leq \frac{3r}{2}\\
				\check C_{\check g}\log|x|& \mbox{if } \frac{3r}{2}<|x| < 2r\\
				0 & \mbox{if } |x| \geq 2r,
			\end{array}
			\right.
		\end{equation}  where $\check  C_{\check g}$ is a positive constant, independent of $r$, to be determined later.\par
		Now for any $r>r_0$ and for all  $x\in B_r\setminus B_{r_0}$, using that $\log|x|$ is a fundamental solution for $(-\De)^s$,   we obtain 	
		\begin{align}\label{car3pp}
			&(-\Delta)^s\check v (x)\notag\\
			&=  C_{1,\frac 12} P.V. \Bigg(\int_{B_{2r}(x)}\frac{\log|x|-\log|x-y|}{|y|^{2}} dy
			+ \int_{\mb R\setminus B_{2r}(x)}\frac{	{\log |x|}}{|y|^{2}} dy\notag\\ &\qquad\qquad\qquad\pm  \int_{\mb R\setminus B_{2r}(x)}
			\frac{{\log{|x|}}-	{\log |x-y|}}{|y|^{2}} dy\Bigg)\notag\\
			&=(-\De)^s(\log|x|)+ C_{1,\frac 12}\int_{\mb R\setminus B_{2r}(x)}\frac{\log |x-y|}{|y|^{2}} dy\notag\\
			&\leq0+ C_{1,\frac 12} \int_{\mb R\setminus B_{2r}(x)}\frac{\log |x-y|}{(\frac 12|x-y|)^{2}} dy \quad \text { (since for $y\in \mb R\setminus B_{2r}(x), \, |y|\geq \frac 12 |x-y|$)} \notag\\
			&\leq  C_5\int_{2r}^{+\infty}\frac{\log t}{t^{2}} dt \qquad\qquad \qquad\qquad\;\text{  (taking $|x-y|=t$)}\notag\\
			&=C_5 \left(\frac{\log{2r}}{2r}+\frac{1}{2r} \right)
			\notag\\&
			\leq 2C_5 \frac{\log{2r}}{r}, \;\text{for sufficiently large $r>r_0$}, 
		\end{align}  where $C_5>0$ is a  constant,  independent of $r$.
		Next,   for all $x\in B_r\setminus B_{r_0},$  we get
		\begin{align}\label{car3pr}
			&(-\Delta)^s\check g(x)\notag\\
			&=C_{1,\frac 12}\left( \int_{{B_{2r}(x)\setminus B_{\frac{3r}{2}(x)}}}\frac{\check g(x)-\check g(x-y)}{|y|^{1+2s}} dy+P.V. \int_{\mb R\setminus\left({B_{2r}(x)\setminus B_{\frac{3r}{2}(x)}}\right)}\frac{\check g(x)-\check g(x-y)}{|y|^{1+2s}} dy\right)\notag\\
			&=C_{1,\frac 12} \int_{B_{2r}(x)\setminus B_{\frac{3r}{2}(x)}}\frac{ -\check C_{\check g}\log|x-y| }{|y|^{2}} dy+0\notag\\
			&\leq C_{1,\frac 12} \int_{B_{2r}(x)\setminus B_{\frac{3r}{2}}(x)}\frac{ -\check C_{\check g}\log|x-y|}{|3r|^{2}} dy\qquad\quad \text{(since  $|y|\leq |x-y|+|x|\leq 2r+ r=3r$)} \notag\\
			&\leq -\check C_{\check g}C_6\, \frac{1}{r^2}\int_{\frac{3r}{2}}^{2r} \log t dt\qquad\quad \text{(taking  $|x-y|=t$)}\notag\\
			&=\check C_{\check g} C_6\,\frac{1}{r^2} \big[t-t \log t \big]_{\frac{3r}{2}}^{2r}\notag\\
			&=\check C_{\check g} C_6\,\frac{1}{r^2}\left[\frac r2-2r\log{2r}+\frac{3r}{2}\log{\frac{3r}{2}}\right]\notag\\
			&=\check C_{\check g} C_6\,\frac{1}{r^2}\left[\frac r2-\frac r2\log{2r}+ \frac{3r}{2}\log{\frac 34}\right]\notag\\
			&\leq \check C_{\check g}C_6\left[\frac {1}{2r}-\frac {\log{2r}}{2r}\right]\notag\\&	
			\leq-{\check C_{\check g}}C_7\frac {\log{2r}}{r}, \;\text{for sufficiently large $r>r_0,$}
		\end{align} where $C_6, C_7$ are two positive constants which are independent of $r$. 
		Now we choose the constant  $\check C_{\check g}$ such that $\check C_{\check g} C_7>2 C_5$, that is, from \eqref{car3pp} and \eqref{car3pr}, for sufficiently large $r>r_0$, it follows that \begin{align}\label{nbbn}(-\De)^s(\check v+\check g)(x)<0, \;\text{ for all  }  x\in B_r\setminus B_{r_0}.\end{align}  For any $r>r_0$, combining  \eqref{cfr2} and \eqref{cfr}, define 
		$$ \check h(x):={\check  v(x)+\check g(x)}
		$$ and	set $$\check \rho(r):=\inf_{B_{2r}\setminus B_{r}}\frac{u}{\check h}>0.$$
		Now  using \eqref{ea2} and  \eqref{nbbn}, for sufficiently large $r>r_0$, we have \begin{align}\label{rfr}(-\De)^s\check \rho(r)\left(\check h-\ds\sup_{\overline B_{r_0}} \check h\right)<0\leq (-\De)^su\quad \text{ in  }   B_{r}\setminus  B_{r_0}.\end{align}   Furthermore, for any $r>r_0$, from the definition of $\check  h$, we find that  $\check \rho(r)\left(\check  h-\ds\sup_{\overline B_{r_0}} \check h\right)\leq0\leq u$ in $\overline {B}_{r_0}$ and 
		$ \check \rho(r)\left(\check h-\ds\sup_{\overline B_{r_0}} \check h\right)\leq u$ in $\mb R\setminus  B_{r}$.   Thus, for any  sufficiently large $r>2r_0,$ applying the  comparison principle to \eqref{rfr}, we get
		\[ \check \rho(r)\left(\check h-\ds\sup_{\overline B_{r_0}} \check h\right) \leq u\quad\text {  in  }   \mb R,\] and taking infimum over $B_{3r_0}\setminus B_{2r_0}$ in the both sides of the last relation, it yields that  \[ 
		\check  \rho(r)\left(  \inf_{B_{3r_0}\setminus B_{2r_0}}\check h-\ds\sup_{\overline B_{r_0}} \check  h\right)\leq \inf_{B_{3r_0}\setminus B_{2r_0}}u.\]
		Thus, the second inequality in \eqref{mbds}  follows. Hence, the proof of the lemma is complete.
	\end{proof}\hfill{\QED}
	{\bf Proof of Theorem \ref{thm2}:}
	We argue the proof by obtaining a contradiction. For that, let us suppose that there is a non-negative function $u$ in $\mb R$ such that,  $u$ is continuous and $u>0$ in $\mb R\setminus B_{r_0},$ for some fixed $r_0> 1$, and  $u$  satisfies the inequality \eqref{slinxs0} in $\mb R\setminus B_{r_0}$.
	For each $r> 2r_0$, denote $$u_r(x) : = u(rx),$$ and it is easy to verify that $u_r$ satisfies the following:
	\begin{equation}\label{ss}
		(-\De)^s u_r\geq r^{2s} f(u_r, rx) \quad \mbox{in} \ \R \setminus B_{\frac12}.
	\end{equation}
	For $r>2r_0$,  set $$m(r) := \inf_{\overline B_{2r}\setminus B_r} u = \inf_{\overline B_2\setminus B_1} u_r.$$  Let $\bar C$ be as in Lemma \ref{vwhlap} with $\nu =\frac12.$
	By  Lemma \ref{vwhlap}, for each $r>2r_0$, for the set $$A_r:=(B_{2} \setminus B_1 ) \cap \{ m(r) \leq u_r \leq \bar C m(r) \},$$ we have  $$|A_r|\geq\frac{1}{2} |B_2\setminus B_1|.$$  Then, applying Proposition \ref{qsmp}-$(i)$, for every $r>2r_0$, it implies that 
	\begin{equation}\label{mrandf}
		{m(r)} \geq c \inf\left\{ r^{2s}f(t,x) : r \leq |x| \leq 2r, \ m(r) \leq t \leq \bar C m(r)\right\},
	\end{equation}where  $ c> 0$ is a positive constant that h is independent of $r$.
	We claim that
	\begin{equation} \label{noloitering}
		\quad m(r) \rightarrow + \infty \quad \mbox{as} \ r \to +\infty.
	\end{equation}
	For that, first observe that since $\sigma^*:=-1+2s>0$ and by Lemma \ref{mbounds_1} and Lemma \ref{mb2}, we have $m(r)\geq \min \{\bar c_{\min}, \underline c_{\min}\}>0,$ for any $r>2r_0$, it follows that $\ds\lim_{r\to+\infty}m(r)\not=0.$ Now suppose that we have a subsequence $r_j \to +\infty$ such that $m(r_j) \rightarrow m_0 \in (0,\infty)$, then by taking $r=r_j \to +\infty$ in \EQ{mrandf}, we find a contradiction to $(f2')$. Thus,  the claim holds.\\ Now,  we aim to complete the proof of Theorem \ref{thm2} by showing \EQ{noloitering}  contradicts  the condition  $(f3')$.
	First, we may assume that $\widetilde\Phi$ is normalized so that $$\ds\max_{ \overline{B}_{r_0}} \widetilde \Phi= 1 \;\; \text{and  } \;\widetilde\Phi> 0 \;\text{in}\; \R \setminus B_{r_0},$$ where $\tilde \Phi=-\Phi$ is as  defined in \eqref{fs} for the case $n=1$ with $\sigma^*:=-1+2s>0$. 
	Using Lemma \ref{mbounds_1} and Lemma \ref{mb2}, we can establish the upper bound of $m(r)$ as 
	\begin{equation} \label{mrupna}
		m(r) \leq \mc C_1 \max_{r\le |x|\le 2r}\widetilde \Phi(x) \leq \mc C \min_{r\le |x|\le 2r}\widetilde \Phi(x), \;\text{ for any } r> 2r_0,
	\end{equation} where $\mc C_1, \mc C>0$ are positive constants,  independent of $r$.
	Next, we find  a lower bound for $m(r)$.
	Recalling $(f3^\prime)$, Let $k>0$ and assume that $r>2r_0$ is large enough  so that  $\underline\mu\leq m(r)$. Suppose for contradiction that $ m(r)\leq k \ds\min_{r\le |x|\le 2r}\widetilde \Phi(x)$. This, combining with \EQ{mrandf} and \eqref{noloitering}, for sufficiently large $r> 2r_0$, implies  that 
	\begin{align*}
		\inf_{r\le |x|\le 2r}\widetilde\Psi_{\bar C k}(x)
		&\leq
		\inf\left\{ t^{-1} |x|^{2s} f(t,x) : r\le |x| \leq 2 r , \ \underline\mu \leq t \leq \bar C k \widetilde \Phi(x) \right\} \\
		& \leq \inf \left\{ t^{-1} r^{2s} f(t,x) : r \leq |x| \leq 2r, \ m(r) \leq t \leq \bar C m(r) \right\} \\
		&\leq \mc C_2\frac {1}{m(r)} \inf \left\{  r^{2s} f(t,x) : r \leq |x| \leq 2r, \ m(r) \leq t \leq \bar C m(r) \right\} \\
		&\leq \mc C_3,
	\end{align*} where $\mc C_2, \mc C_3>0$ are positive constants  not  depending on $r$.
	This contradicts $(f3')$ if $k> 0$ is  sufficiently small, and thus, for any $r>2r_0,$ we obtain 
	\begin{equation}\label{jjj}
		m(r)\geq \tilde c_1 \min_{r\le |x|\le 2r} \widetilde\Phi(x)\geq \tilde c\max_{r\le |x|\le 2r} \widetilde\Phi(x), 
	\end{equation} where $\tilde c_1, \tilde c>0$ are some  positive constants not depending on $r$.
	Using \EQ{mrupna} and \eqref{jjj}, we deduce
	\begin{equation}\label{mrtrap2}
		\tilde c\max_{r\le |x|\le 2r}\widetilde \Phi(x) \leq m(r) \leq \mc C \min_{r\le |x|\le 2r}\widetilde \Phi(x),  \quad \mbox{for sufficiently large}  \ r>2r_0,
	\end{equation} where $\tilde c, \mc C$ are  as given in \eqref{mrupna} and in \eqref{jjj}, respectively.
	Define the quantity
	\begin{equation*}
		\tilde\eta(r) : = \inf_{\mb R\setminus B_r} \frac{u}{\widetilde\Phi -1}>0, \quad \text{ for any }r > r_0.
	\end{equation*} Then,  for any $r>r_0$, it holds that
	\begin{align}\label{gbf}
		\tilde\eta(r)({\tilde \Phi}-1)=0\leq u \;\;\text{ in } \overline B_{r_0}.\end{align}and for every sufficiently large $r>r_0$, we have
	\begin{align}\label{gbf2}
		\tilde\eta(r)({\tilde \Phi}-1)\leq u \;\;\text{ in }\mb R\setminus  B_{r}.\end{align}
	Moreover, for any $r>r_0,$ we get \begin{align}\label{gbf3}  (-\De)^s\tilde\eta(r)({\tilde \Phi(x)}-1)=0\leq f(u,x) \leq (-\De)^s u(x),\; \mbox{ for  all }  x\in B_r \setminus B_{r_0}.\end{align}
	Therefore, for sufficiently large $r>r_0$,  using the comparison principle to  \eqref{gbf}, \eqref{gbf2} and \eqref{gbf3}, we obtain
	\begin{equation}\label{pns4}
		\tilde\eta(r) \left( \widetilde \Phi - 1 \right)\leq u\quad \mbox{in} \;\;\;\mb R.
	\end{equation}
	Next, for $r> 2r_0$, set the positive functions
	\begin{equation*}
		v_r(x): = u_r(x) +\tilde \eta(2r), \qquad w_r(x):=\tilde \eta(2r)  \widetilde \Phi(rx).
	\end{equation*}
	Clearly, for sufficiently large $r>2r_0$, by  \eqref{pns4}, it holds that  $v_r(x)\ge w_r(x)$ and recalling \eqref{ss} and the fact that $\tilde \Phi$ is a fundamental solution for $(-\De)^s$, we infer that
	\begin{equation}\label{sss}
		(-\De)^s v_r (x)\geq r^{2s} f(u_r,rx)\ge 0 = (-\De)^sw_r(x), \;\mbox{ for all }\  x\in B_4 \setminus B_{\frac 12}.
	\end{equation}
	Now	using \EQ{mrandf}, \EQ{mrtrap2} and $(f3')$, for every sufficiently large $r> 2r_0$, it follows that
	\begin{align*}
		\inf_{A_r}\left( r^{2s} f(u_r,rx) \right) & \geq \frac{1}{2^{2s}} \inf \left\{ |y|^{2s} f(t,y) : r \leq |y| \leq 2r, \ m(r) \leq t \leq \bar C m(r) \right\} \\
		& \geq \tilde c_2 {m(r)} \inf\left\{ t^{-1} |y|^{2s} f(t,y) : \underline\mu \leq t \leq \mc C \bar C\widetilde\Phi(y), \ r \leq |y| \right\} \\
		&  \geq \tilde c_3 {m(r)}, 
	\end{align*} where $\tilde c_2, \tilde c_3>0$ are positive constants which do not depend on $r$ and thus, 
	in particular, for every such large $r>2r_0$, by  \eqref{sss}, we have
	\begin{equation*}
		(-\De)^s (v_r -w_r)(x)=	(-\De)^s v_r(x)\geq \tilde c_3{m(r)} \chi_{A_r}(x), \; \mbox{ for all  } \ x\in B_{4} \setminus B_{\frac 12}.
	\end{equation*}
	Hence, from the last relation, applying Proposition \ref{qsmp}-$(i)$ with $A:=A_r$, $\Om:=B_4 \setminus B_{\frac 12}$ and $K:=\overline B_2 \setminus B_{1}$,  for every  large $r>2r_0,$ we get
	\begin{equation}\label{kaj}
		v_r(x) \geq \tilde c_4 m(r) +w_r (x),\; \mbox{ for   all } \ x\in\overline  B_2\setminus B_{1},
	\end{equation}
	where $\tilde c_4> 0$ is some positive constant that  is  independent  of $r$. Recalling the definition of $v_r$ and $w_r$, combining with \re{mrtrap2} and \eqref{kaj},  for every sufficiently large $r> 2r_0$, we deduce
	\begin{equation}\label{dm,}
		u(x) - \tilde\eta(2r) \left( \widetilde \Phi(x) - 1 \right) \geq \tilde c_4 m(r) \geq \tilde c_0 \widetilde\Phi (x),\;\mbox{ for all } \ x\in \overline B_{2r}\setminus B_r,
	\end{equation}  	where $\tilde c_0> 0$ is a constant  not depending on $r$. Now by \eqref{pns4}, we have, for all $r>r_0$,  $$\tilde\eta(r)=:C_0,$$ where $C_0$ is some positive constant which does not depend on $r$. Therefore, there exists a sequence $\{x_j\}\subset \mb R$ such that by the definition of $\tilde\eta$, it yields that
	\begin{equation}\label{dm,,}
		\frac{u(x_j)}{\widetilde\Phi(x_j) -1}\to C_0\; \;\;\;\text{ as } |x_j|\to+\infty.
	\end{equation}  Choose $r_j:=\frac 23|x_j|$ such that $r_j:=\frac 23|x_j|\leq |x_j|\leq \frac 43|x_j|=:2r_j.$ So, $x_j\in \overline B_{2r_j}\setminus B_{r_j}$
	and hence, from \eqref{dm,}, we obtain
	\begin{equation*}
		u(x_j) - \left( C_0 + \tilde c_0 \right) \left( \widetilde\Phi (x_j)- 1 \right) > 0.
	\end{equation*} In the last inequality, passing to the limit $|x_j|\to+\infty$ and then using \eqref{dm,,}, we obtain $-c_2\geq 0$, where $c_2$ is as  in \eqref{dm,}. This  is absurd.
	Thus, the proof is  complete.\hfill{\QED}
	
	\section{ Liouville type Theorem: Case $2s<n$} \label{sectwo}
	In this section,  in the case $2s<n$, first we will show  the  nonexistence of positive solutions of the following fractional   inequality:
	\begin{equation} \label{slinxs}
		(	-\Delta)^s u \ge f(u)
	\end{equation}
	in exterior domains of $\mb R^n$,  $n\geq 1$.  The nonlinearity $f=f(t)$ satisfies the following hypotheses:
	{	\begin{itemize}
			\item[$(f1)$] The function $f:(0, +\infty) \to (0,+\infty)$ is continuous and positive.
			\item[$(f2)$] We assume that 
			\begin{equation} \label{fpow}
				\liminf_{t\to 0} t^{-\frac{n}{n-2s}}f(t) >0.
			\end{equation}
	\end{itemize}}
	\begin{Example}
		An  example of such function $f$ satisfying $(f1)$, $(f2)$ is given  below:
		\begin{equation}\label{cf1b}f (t)  :=
			\left\{
			\begin{array}{ll}
				\frac52\, t^{\frac{n}{n-2s}} & \mbox{if } t \leq 1 \\
				\cos{2\pi t}+1+\frac 1t& \mbox{if } t \geq 1.
			\end{array}
			\right.
		\end{equation}
	\end{Example}
	\noi	The statement of the next main theorem  in this article reads as follows:
	\begin{Theorem}\label{slthm}
		Let $s\in(0,1)$, $2s<n$ and assume that the nonlinearity $f$ satisfies the hypotheses $(f1)$-$(f2)$. Then the  inequality \re{slinxs} has no continuous positive {solution} in any exterior domain of $\mb R^n, n\geq1$.
	\end{Theorem}
	
	\begin{Remark}
		In this section, we explicitly prove Theorem \ref{slthm}, the analogous result to Theorem \ref{thm2}, in the case  $2s<n$, $n\geq1$ with  the nonlinearity $f(u)$ satisfying the assumptions $(f1)$-$(f2)$, for the  sake of simplicity. However, we would like to mention that adopting  the same  notion as the proof of  Theorem \ref{thm2},  one can be able to extend the result in   Theorem \ref{slthm}   for general $f:=f(u,x)$. The precise statement of the said result is given in the next theorem.
	\end{Remark}
	\begin{Theorem}\label{thm3}	Let $s\in (0,1)$ with $2s<n$ and $n\geq 1$. Suppose the nonlinearity $f=f(t,x)$	such that
		\begin{align}\label{aq}
			f: (0, +\infty) \times (\R^n \setminus B_{r_0}) \to (0,+\infty) \text{ is continuous for any $r_0>0$.}
		\end{align} 
		Moreover, $f$ satisfies $(f2^\prime)$ in addition to the following hypothesis:
		\begin{itemize}
			\item[$(f4^\prime)$] Let $\sigma^*:=-n+2s <0$. Then, there exists a constant $\bar \mu > 0$ such that  if we define
			\begin{equation*}
				\Psi_k(x) := |x|^{2s} \inf_{k\Phi(x)\le t\le\bar \mu }  \frac{f(t,x)}{t} \quad\mbox{and}\quad h(k):=  \liminf_{|x|\to +\infty} \Psi_k(x),
			\end{equation*}
			then $0 < h(k) \leq +\infty,$ for each $k>0$, and
			$	\ds\lim_{k\to +\infty} h(k) =+\infty.$
		\end{itemize}
		Then, the inequality \re{slinxs0} has no continuous positive  solution in any exterior domain of $\mb R^n$.
		
	\end{Theorem}
	We present below some models  of such nonlinearities $f$ that satisfy the hypotheses \eqref{aq}, $(f2^\prime), (f4^\prime)$. We encourage the readers to see \cite{chen2} for more such examples.
	\begin{Remark}
		\begin{itemize}
			\item[1.] Consider the nonlinearity $f$ of the form \eqref{gbz}. Then clearly, \eqref{aq} and $(f2^\prime)$  imply that $g:(0,+\infty)\to(0,+\infty)$ is continuous.  We claim that, together with this,  a sufficient condition for $(f4^\prime)$ is:
			\begin{align}\label{z3} \text {if $ 2s<n$, {then}  $\ds \liminf_{t\to 0} t^{- \widetilde\al^*}g(t) > 0$,  where  $ \widetilde\al^*:= 1+ \frac{2s-\gamma}{-\sigma^*}.$  }
			\end{align}
			Arguing in a similar fashion as in Remark \ref{grmk}, we can obtain the claim in \eqref{z3}.
			\item[2.] Consider a function $f$ satisfying \eqref{aq} and $(f2^\prime)$ such that $\frac{f(t,x)}{t}$ is increasing in $t>0$, for all $x\in \mb R^n\setminus B_{r_0}$.
			Moreover, for each $k>0$,
			\[|x|^nf(k|x|^{-n+2s},x)>z_2(k),\] where $\ds\lim_{k \to +\infty} \frac{z_2(k)}{k}=+\infty.$ One of the standard examples for $z_2$ is  $z_2(k)=a_2 k^b$,$b>1$, with some  constant $a_2>0$. The above condition implies $(f4^\prime)$. \end{itemize}
	\end{Remark}
	
	\noi Before proving Theorem \ref{slthm}, we establish the following important lemmas.
	\begin{Lemma}\label{fundycmpl}
		Let  $s\in (0,1)$  with $ 2s<n$ and $\Om$	 be an exterior domain  of $\R^n$, $n\geq 1$.	Let $u$  be a non-negative function in $\mb R^n$ such that,  $u$ is  continuous and $u>0$     in $\Om$ and moreover, \begin{align}\label{niru}(-\De)^su\geq 0 \text  { in } \Om.\end{align} Then, there are constants $c_{\min},\,C_{M}> 0$, depending only on $u, s, n$ and $\Omega$, such that
		\begin{equation} \label{fundycmp}
			c_{\min} r^{-n+2s} \leq \inf_{B_{2r} \setminus B_r} u\leq C_M \quad \mbox{for every sufficiently large} \ r> 0.
		\end{equation}
	\end{Lemma}
	\begin{proof}
		Define a function \begin{equation}\label{cf1}\hat \Psi (x)  :=
			\left\{
			\begin{array}{ll}
				1  & \mbox{if } |x| \leq 1 \\
				|x|^{\sigma^* }& \mbox{if } |x| \geq 1,
			\end{array}
			\right.
		\end{equation} where $\sigma^*=-n+2s<0$.  
		Then, for every $x\in \mb R^n\setminus  B_{1}$, using the fact that $|x|^{\sigma^*}$ is a fundamental solution for $(-\De)^s$,	we calculate the following:
		\begin{align}\label{ca1f}
			&(-\Delta)^s\hat \Psi(x)\notag\\
			&=C_{n,s}\left(P.V. \int_{\mb R^n\setminus B_1(x)}\frac{\hat \Psi(x)-\hat \Psi(x-y)}{|y|^{n+2s}} dy+ \int_{ B_1(x)}\frac{\hat \Psi(x)-\hat \Psi(x-y)}{|y|^{n+2s}} dy\right)\notag\\
			&= C_{n,s}\Bigg(P.V.  \int_{\mb R^n\setminus B_1(x)}\frac{|x|^{\sigma^*}-|x-y|^{\sigma^*}}{|y|^{n+2s}} dy+\int_{ B_1(x)}\frac{|x|^{\sigma^*}-1}{|y|^{n+2s}} dy
			\pm \int_{ B_1(x)}\frac{|x|^{\sigma^*}-|x-y|^{\sigma^*}}{|y|^{n+2s}} dy\Bigg)\notag\\
			&=(-\De)^s(|x|^{\sigma^*}-1)+C_{n,s}	\int_{ B_1(x)}\frac{|x-y|^{\sigma^*}-1}{|y|^{n+2s}} dy\notag\\
			&\leq0+	C_{n,s}\int_{ B_1(x)}\frac{|x-y|^{\sigma^*}}{(|x|-1)^{n+2s}} dy\ \;\;\qquad \text{ (since $|y|\geq |x|-|y-x|\geq|x|-1$)}\notag\\
			&=C_8(s,n)\int_{0}^{1}\frac {t^{\sigma^*+n-1}}{(|x|-1)^{n+2s}} dt \;\;\qquad\qquad \text{  (taking $|x-y|=t$)}\notag\\
			&=C_9(s,n) \frac{1}{(|x|-1)^{n+2s}},
		\end{align} where $C_8(s,n), C_9(s,n)$ are two positive constants. Next, we  define 
		\begin{equation}\label{cf1h}\tilde\gamma (x)  :=
			\left\{
			\begin{array}{ll}
				1  & \mbox{if } |x| <1 \\
				0& \mbox{if } |x| \geq1.
			\end{array}
			\right.
		\end{equation} Then, for $x\in \mb R^n\setminus B_1,$  we obtain
		\begin{align}\label{ca1aa}
			&(-\Delta)^s \tilde\gamma(x)\notag\\
			&=C_{n,s}\left(P.V. \int_{\mb R^n\setminus B_1(x)}\frac{ \tilde\ga(x)- \tilde\ga(x-y)}{|y|^{n+2s}} dy+ \int_{ B_1(x)}\frac{\tilde \ga(x)-\tilde \ga(x-y)}{|y|^{n+2s}} dy\right)\notag\\
			&=0+C_{n,s}	\int_{ B_1(x)}\frac{-1}{|y|^{n+2s}} dy\notag\\
			&\leq C_{n,s }	\int_{ B_1(x)}\frac{-1}{(1+|x|)^{n+2s}} dy\ \;\;\qquad \text{ (since $|y|\leq |x-y|+|x|\leq1+|x|$)}\notag\\
			&=C_{\ga}(s,n)\frac {-1}{(1+|x|)^{n+2s}},
		\end{align}
		where $C_{\ga}(s,n)>0$ is a positive constant. Now set \[\hat\Psi_\ga(x):= \hat \Psi(x)+\frac{2C_9(s,n)}{C_\ga(s,n)}\tilde\ga(x),
		\] where $C_9(s,n)$ and $C_{\ga}(s,n)$ are defined in \eqref{ca1f} and in \eqref{ca1aa}, respectively. Then for any $r>1$, 
		\begin{align}\label{lld}
			c(r)\hat\Psi_\ga(x)\leq u(x),\;\text{ for all  } x\in \overline B_r,
		\end{align} 
		where $c(r):=\ds\frac{C_\ga(s,n)}{4C_9(s,n)}\ds\inf_{B_r}u>0.$
		On the other hand, for any $\e>0$ and for any $r>1$, there exists $R_*:=R_*(\e,r)>r,$ such that
		\begin{align}\label{llde}
			c(r)\hat\Psi_\ga(x)\leq u(x)+\e,\;\text{ for all } x\in  \mb R^n\setminus B_{R_*}.
		\end{align} 
		
		Now by the definition of $\hat\Psi_\ga$, \eqref{ca1f} and \eqref{ca1aa}, there exists a  large $r_0>1$ with $\mb R^n\setminus B_{r_0}\subset \Om,$ such that
		\begin{align}\label{vask}
			(-\Delta)^s	c(r)\hat\Psi_\ga(x)&\leq c(r)	C_9(s,n)\left( \frac{1}{(|x|-1)^{n+2s}}-\frac{2}{(1+|x|)^{n+2s}}\right)\notag\\&<0, \;\text{ for all  $x\in\mb R^n\setminus B_{r_0},$}	\end{align} that is, in particular taking $r=r_0$ and using \eqref{niru}, we get \begin{align}\label{vask1}(-\De)^sc(r_0)\hat\Psi_\ga(x)<0\leq (-\De)^su(x), \;\text{ for all } x\in \mb R^n\setminus B_{r_0}.\end{align} So, in particular,  by choosing  $r=r_0$ in \eqref{lld} and in  \eqref{llde} and applying the comparison principle to \eqref{vask1},  we obtain\[  c(r_0)\hat\Psi_\ga(x)\leq u(x)+\e, \;\text{ for all  } x\in \mb R^n.\] Passing to the limit $\e\to0$ in the last relation, it follows that
		\begin{align}\label{jjjj}c(r_0)\hat\Psi_\ga(x)\leq u(x), \;\text{ for all  } x\in \mb R^n.\end{align}
		Now for any large $r>r_0$, taking infimum over $ B_{2r}\setminus B_r$ in the both sides of \eqref{jjjj}, we find  the first inequality in \EQ{fundycmp}. \\
		To achieve the second inequality in \EQ{fundycmp},  for any $r>r_0$,  first we define a function 
		\begin{equation}\label{cfp}\Psi (x)  :=
			\left\{
			\begin{array}{ll}
				1-	r_0^{-\sigma^*}|x|^{\sigma^*} & \mbox{if } |x| < 2r \\
				0 & \mbox{if } |x| \geq2r, 
			\end{array}
			\right.
		\end{equation}where $\sigma^* :=-n+2s<0$. In addition to that,  for any $r>r_0$, we define another function  \begin{equation}\label{cf2} g (x)  =
			\left\{
			\begin{array}{ll}
				0  & \mbox{if } |x|\leq \frac{3r}{2}\\
				C_g& \mbox{if } \frac{3r}{2}<|x| < 2r\\
				0 & \mbox{if } |x| \geq 2r,
			\end{array}
			\right.
		\end{equation}  where $C_g$ is a positive constant,  independent of $r$, to  be determined later.\\
		Now for any $r>r_0$ and for all $x\in B_r\setminus B_{r_0},$ using the fact that $-|x|^{\sigma^*}$ is a fundamental solution for $(-\De)^s$, we estimate the following:
		\begin{align}\label{ca3q}
			&(-\Delta)^s\Psi (x)\notag\\
			&= C_{n,s}\left(P.V.\int_{B_{2r}(x)}\frac{ \Psi(x)-\Psi(x-y)}{|y|^{n+2s}} dy+ \int_{\mb R^n\setminus B_{2r}(x)}\frac{ \Psi(x)-\Psi(x-y)}{|y|^{n+2s}} dy\right)\notag\\
			&= C_{n,s} \Bigg(P.V.\int_{B_{2r}(x)}\frac{(1-r_0^{-\sigma^*}|x|^{\sigma^*})-(1-r_0^{-\sigma^*}|x-y|^{\sigma^*})}{|y|^{n+2s}} dy
			+ \int_{\mb R^n\setminus B_{2r}(x)}\frac{1-r_0^{-\sigma^*}|x|^{\sigma^*}}{|y|^{n+2s}} dy\notag\\&\qquad\pm  \int_{\mb R^n\setminus B_{2r}(x)}\frac{(1-r_0^{-\sigma^*}|x|^{\sigma^*})-(1-r_0^{-\sigma^*}|x-y|^{\sigma^*})}{|y|^{n+2s}} dy\Bigg)\notag\\
			&=(-\De)^s(1-r_0^{-\sigma^*}|x|^{\sigma^*})+C_{n,s}\int_{\mb R^n\setminus B_{2r}(x)}\frac{1-r_0^{-\sigma^*}|x-y|^{\sigma^*}}{|y|^{n+2s}} dy\notag\\
			&\leq 0+C_{n,s}\int_{\mb R^n\setminus B_{2r}(x)}\frac{1-r_0^{-\sigma^*}|x-y|^{\sigma^*}}{(\frac 12|x-y|)^{n+2s}} dy \qquad \text {( since for $y\in \mb R^n\setminus B_{2r}(x), \, |y|\geq \frac 12 |x-y|$)} \notag\\
			&\leq C_{10}(s,n) \int_{2r}^{+\infty}\frac{(1-r_0^{-\sigma^*}t^{\sigma^*})t^{n-1}}{t^{n+2s}} dt \qquad\qquad\text{  (taking $|x-y|=t$)}\notag\\
			&=C_{10}(s,n) \left(\frac {(2r)^{-2s}}{2s}+ r_0^{-\sigma^*}\frac{(2r)^{\sigma^*-2s}}{\sigma^*-2s}\right)
			\notag\\
			&
			\leq C_{11}(s,n) \frac{1}{r^{2s}},
		\end{align}  
		where $C_{10}(s,n), C_{11}(s,n)$ are two positive constants.
		Next, for any $r>r_0$  and for all $x\in B_r\setminus B_{r_0},$ we have 	
		\begin{align}\label{ca3p}
			&(-\Delta)^sg(x)\notag\\
			&= C_{n,s}\left(\int_{{B_{2r}(x)\setminus B_{\frac{3r}{2}(x)}}}\frac{ g(x)- g(x-y)}{|y|^{n+2s}} dy+P.V. \int_{\mb R^n\setminus\left({B_{2r}(x)\setminus B_{\frac{3r}{2}(x)}}\right)}\frac{ g(x)- g(x-y)}{|y|^{n+2s}} dy\right)\notag\\
			&= C_{n,s}\int_{B_{2r}(x)\setminus B_{\frac{3r}{2}(x)}}\frac{ -C_g}{|y|^{n+2s}} dy+0\notag\\
			&\leq C_{n,s} \int_{B_{2r}(x)\setminus B_{\frac{3r}{2}}(x)}\frac{ -C_g}{|3r|^{n+2s}} dy\;\quad \text{(since $ |y|\leq |x-y|+|x|\leq 3r$)} 
			\notag\\
			&	={ -C_g}C_{12}(s,n)\frac{1}{r^{2s}}, 
		\end{align}  where $C_{12}(s,n)>0$ is a positive constant.
		Then, we choose the constant  $C_g$ such that ${ C_g}C_{12}(s,n)>C_{11}(s,n)$, that is, from \eqref{ca3q} and \eqref{ca3p}, it follows that  \begin{align}\label{nitu}(-\De)^s(\Psi+g)(x)<0, \;\;\text{for  all } x\in B_r\setminus B_{r_0}.\end{align} Now any  for $r>r_0$, using \eqref{cfp} and \eqref{cf2}, define the function
		$$ \Psi_g(x):=\frac {\Psi(x)+g(x)}{1+C_g},
		$$  where $C_g$ is as  chosen for  \eqref{nitu}. Clearly $\Psi_g(x)\leq 1.$ So,  combining \eqref{niru} and \eqref{nitu}, for every  $r>r_0$, we find that \begin{align}\label{pe}(-\De)^s\left(\inf_{B_{2r}\setminus B_{r}} u\right)\Psi_g(x)<0\leq(-\De)^su(x),\;\; \text{for all }   x\in B_{r}\setminus  B_{r_0}.\end{align} Moreover, for every $r>r_0,$ we have $\left(\ds\inf_{B_{2r}\setminus B_{r}} u\right)\Psi_g\leq u$ in $\overline {B}_{r_0},$ since $\Psi_g<0$ in $\overline{B}_{r_0}$; and $\left(\ds\inf_{B_{2r}\setminus B_{r}} u\right)\Psi_g\leq u$ in $\mb R^n\setminus  B_{r}$,  since $\Psi_g\leq 1$ in $B_{2r}\setminus B_{r}$ and  $\Psi_g=0$ in $\mb R^n\setminus B_{2r}$. Thus, applying the  comparison principle to \eqref{pe}, for every sufficiently large  $r>2r_0$, we get
		\[ \left(\ds\inf_{B_{2r}\setminus B_r} u\right)\Psi_g\leq u \text {  in  }   \mb R^n,\] and then taking infimum over ${B_{3r_0}\setminus B_{2r_0}}$  in  both sides of the above relation, we infer that  	\[ \ds\left(\ds\inf_{B_{2r}\setminus B_r} u\right)\inf_{{B_{3r_0}\setminus B_{2r_0}}}\Psi_g\leq \ds\inf_{{B_{3r_0}\setminus B_{2r_0}}}u.\]This  gives the second inequality  in  \eqref{fundycmp}. Hence, the proof is complete.
	\end{proof}\hfill{\QED}
	\noi	{\bf Proof of Theorem \ref{slthm} :}
	Fix $r_0>1$  very large.  Let us suppose that there exists a non-negative solution $u$ to the inequality \eqref{slinxs} in $\mb R^n\setminus B_{r_0}$ such that, $u$ is continuous and $u>0$ in $\mb R^n\setminus B_{r_0}$.  Next,			for each $r> 2r_0$, set $$u_r (x) := u(rx).$$  Observe that $u_r$ satisfies the following:
	\begin{equation}\label{ur}
		(-\Delta)^s u_r \geq r^{2s }f(u_r) \quad \mbox{in} \ \R^n \setminus B_{\frac{r_0}{r}}.
	\end{equation} 
	For each $r> 2r_0$, as in the previous section, we define
	\begin{equation*}
		m(r): = \inf_{\bar B_2\setminus B_1} u_r = \inf_{\bar B_{2r} \setminus B_r} u.
	\end{equation*}
	Let us set $A_r:=\{x\in B_2\setminus B_1: m(r)\le u_r(x)\le \bar C m(r)\}$, $r>2r_0,$ where $\bar C = \bar C(s,n,\frac{1}{2})> 1$ is a constant  given as in Lemma \ref{vwhlap}. Then from Lemma \ref{vwhlap}, we  have
	\begin{equation*}
		|A_r|\ge \frac 12|B_2\setminus B_1|.
	\end{equation*}
	Hence,  using Proposition \ref{kslap} by choosing $A=A_r$ and $h(x):= r^{2s }f(u_r(x))$, we obtain
	\begin{align}\label{limmins1}
		m(r)  \geq \frac{1}{2} \bar cr^{2s} |B_2\setminus B_1|\min_{t\in \left[ m(r),\bar C m(r)\right]} f(t), \quad \mbox{for every} \ r> 2r_0,
	\end{align}
	where the constant  $\bar c>0$ is as in Proposition \ref{kslap},  which does not depend on $r$. Then \eqref{limmins1} and  Lemma \ref{fundycmpl} yield that 
	\begin{equation} \label{limmins}
		\min_{t\in \left[ m(r),\bar C m(r)\right]} f(t)\le \frac{2}{\bar c}r^{-2s}m(r)\le \frac{2}{\bar c} C_Mr^{-2s} \to 0\quad\mbox{ as }\; r\to +\infty,
	\end{equation} where $C_{M} $ is as given in \eqref{fundycmp}.
	Moreover, using  the continuity and positivity  of the function $f$ and the fact that $m(r)\le C_M$ by \eqref{fundycmp}, it follows that
	\begin{align}\label{m0}m(r) \to 0 \text {   as } r\to +\infty.
	\end{align}
	Hence, for sufficiently large $r>0$, combining  \EQ{fpow} and \EQ{limmins}, we infer that
	\begin{equation}\label{limmins2}
		(m(r))^{\frac{n}{n-2s}} \le \min_{t\in \left[ m(r),\bar C m(r)\right]} f(t)\le  \frac{2}{\bar c} C_M r^{-2s}m(r),
	\end{equation} that is,
	\begin{equation} \label{mupbnd}
		m(r) \leq \frac{2}{\bar c} C_M r^{-n+2s}, \quad \mbox{for every sufficiently large} \ r> 2r_0.
	\end{equation}  Thus, from \eqref{fundycmp} and \eqref{mupbnd}, we have	\begin{equation} \label{fundycmp.2}
		c_{\min} r^{-n+2s} \leq \inf_{ B_{2r} \setminus B_r} u \leq C_{\max}r^{-n+2s}, \quad \mbox{for every sufficiently large} \ r> 0,
	\end{equation} where $c_{\min} $ is as given in \eqref{fundycmp} and $C_{\max}$ is  a positive constant,   not depending on   $r$. On the other hand, using \eqref{limmins1}, \eqref{limmins2} and \eqref{fundycmp.2}, for every $r>2r_0$, we achieve
	\begin{align}\label{lkj}
		\frac{1}{2} \bar cr^{2s} |B_2\setminus B_1|\min_{t\in \left[ m(r),\bar C m(r)\right]} f(t)&\geq \frac{1}{2} \bar cr^{2s} |B_2\setminus B_1|(m(r))^{\frac{n}{n-2s}}\notag\\&\geq \frac{\bar c c_{\min} }{2} |B_2\setminus B_1| r^{-n+2s}\nonumber\\&=:ar^{-n+2s}.
	\end{align}
	Next,  for any $r>2r_0,$ define the following  functions:
	\begin{equation}\label{cf15} \hat w (x)  :=
		\left\{
		\begin{array}{ll}
			0 & \mbox{if } |x| <r \\
			|x|^{\sigma^* }& \mbox{if } |x| \geq  r;
		\end{array}
		\right.
	\end{equation}   \begin{equation}\label{cf155} \hat \ga(x)  :=
		\left\{
		\begin{array}{ll}
			0 & \mbox{if } |x| <r\\
			\mu |x|^{\sigma^*} & \mbox{if }  r\leq|x| \leq \frac {3r}{2}\\
			0& \mbox{if } |x| >\frac {3r}{2},
		\end{array}
		\right.
	\end{equation} 
	where $\sigma^*=-n+2s<0$ and $\mu>0$ is a positive constant to be chosen later.
	Now for any $r>2r_0,$ set the quantity
	\begin{equation*}
		\rho(r) : = \inf_{B_{2r}\setminus B_r} \frac{u}{w_\ga} > 0, 
	\end{equation*}  where 
	\[w_{\ga}(x):= \hat w(x)+\hat\ga(x).\]
	{Note that $ \rho(r)$ is bounded above for all $r>2r_0$. Indeed, by \eqref{fundycmp.2} and the definition of $w_\ga$, 
		\begin{align}\label{cmax}
			\rho(r)\leq 	\inf_{B_{2r}\setminus B_{r}} \frac{u}{w_\ga}
			\leq 
			\frac{\ds\inf_{B_{2r}\setminus B_{r}} u}{\ds\inf_{B_{2r}\setminus B_{r}} w_\ga} 
			\leq 
			\frac{C_{\max}r^{\sigma^*}}{(2r)^{\sigma^*}}
			=: C^\prime, \end{align} where $C_{\max}>0$   is   as in \eqref{fundycmp.2} and $C^\prime$ is a positive constant not depending on $r$.} 
	Next,  for  all $x\in \mb R^n\setminus B_{2r}$ with $r>2r_0$, using  that $|x|^{\sigma^*}$ is a fundamental solution for $(-\De)^s$,  we deduce
	\begin{align}\label{ca10}
		&(-\Delta)^s\hat w (x)\notag\\
		&= C_{n,s}\left(\int_{B_{r}(x)}\frac{ \hat w(x)- \hat w(x-y)}{|y|^{n+2s}} dy+P.V. \int_{\mb R^n\setminus B_{r}(x)}\frac{ \hat w(x)- \hat w(x-y)}{|y|^{n+2s}} dy\right)\notag\\
		&= C_{n,s}\left(\int_{B_{r}(x)}\frac{ |x|^{\sigma^*}}{|y|^{n+2s}} dy+P.V. \int_{\mb R^n\setminus B_{r}(x)}\frac{ |x|^{\sigma^*}- |x-y|^{\sigma^*}}{|y|^{n+2s}} dy\pm \int_{ B_{r}(x)}\frac{ |x|^{\sigma^*}- |x-y|^{\sigma^*}}{|y|^{n+2s}} dy\right)\notag\\
		&=C_{n,s}	\int_{ B_{r}(x)}\frac{|x-y|^{\sigma^*}}{|y|^{n+2s}} dy+(-\De)^s|x|^{\sigma^*}\notag\\
		&\leq C_{n,s}	\int_{ B_r(x)}\frac{|x-y|^{\sigma^*}}{(|x|-r)^{n+2s}} dy+0\ \;\;\qquad \text{ (since $|y|\geq |x|-|y-x|\geq|x|-r$)}\notag\\
		&=C_{13}(s,n)\int_{0}^{r}\frac {t^{\sigma^*+n-1}}{(|x|-r)^{n+2s}} dt \;\;\qquad\qquad \text{  (taking $|x-y|=t$)}\notag\\
		&=C_{14}(s,n) \frac{r^{2s}}{(|x|-r)^{n+2s}},
	\end{align} where $C_{13}(n), C_{14}(s,n)$ are  positive constants.  Furthermore,  for any  $r>2r_0$ and    for  all $x\in \mb R^n\setminus B_{2r}$, using \eqref{ca10}, we get that
	\begin{align}\label{ca10l}
		&(-\Delta)^s\hat \ga (x)\notag\\
		&= C_{n,s}\left(\int_{ B_{\frac {3r}{2}}(x)\setminus B_{r}(x)}\frac{ \hat \ga(x)- \hat \ga(x-y)}{|y|^{n+2s}} dy+P.V. \int_{\mb R^n\setminus\left(B_{\frac {3r}{2}}(x)\setminus B_{r}(x)\right)}\frac{ \hat \ga(x)- \hat \ga(x-y)}{|y|^{n+2s}} dy\right)\notag\\
		&=C_{n,s}\int_{ B_{\frac {3r}{2}}(x)\setminus B_{r}(x)}\frac{-\mu|x-y|^{\sigma^*}}{|y|^{n+2s}} dy+0
		\notag\\
		&\leq C_{n,s}	\int_{ B_{\frac {3r}{2}}(x)\setminus B_{r}(x)}\frac{-\mu|x-y|^{\sigma^*} }{(|x|+2r)^{n+2s}} dy\; \text{ (since $|y|\leq |x-y|+|x|\leq 2r+|x|$)}\notag\\
		&	=C_{15}(s,n)\int_{3r}^{4r}\frac{-\mu t^{\sigma^*+n-1} }{(|x|+2r)^{n+2s}}dt\; \text{ (taking $|x-y|=t$ )}\notag\\
		&	\leq- C_{16}(s,n)\frac{\mu r^{2s}}{(|x|+2r)^{n+2s}},
	\end{align} where $C_{15}(s,n)$ and $C_{16}(s,n)$ are some positive constants. Therefore, using \eqref{ca10} and \eqref{ca10l}, choose $\mu>0$ such that $C_{16}(s,n)\mu>2 C_{14}(s,n)$ and then   there exists  a large $  R_0:=  R_0( s,n,\mu, C_{14}(s,n), C_{16}(s,n), r_0 )>2r_0$ such that $	(-\Delta)^s(\hat w+\hat \ga)(x)<0,$ for all $|x|>  R_0.$ Thus,  for all   $r>2 R_0$, we have
	\begin{align}\label{ri}(-\De)^s\rho(r)w_\ga<0<f(u)\leq (-\De)^s\;\;\text{in  }\mb R^n\setminus B_{r}.\end{align}
	Observe from the definition of $w$ and $\rho$ that, for any $r>2R_0$, it hold that
	\begin{align}\label{lldb}
		\rho(r)w_\ga\leq 	u\;\;\text { in }  \overline B_{r}.
	\end{align} 				 
	\noi	On the other hand, for any given  $\e>0$ and $r>2R_0$,  there exists a very large  $R:=R(\e, R_0,r)>r$,  such that  
	\begin{align}\label{lldeh}
		\rho(r)w_\ga\leq 	u+\e\;\; \text{ in }   \mb R^n\setminus B_{R}.
	\end{align} 
	\noi	 So,  for  any sufficiently large $r> 2R_0$, applying the comparison principle to  \eqref{ri},   \eqref{lldeh} and  \eqref{lldb}, 
	we obtain
	\[ \rho(r)w_\ga(x)\leq 	u(x)+\e, \;\text{ for all  } x\in \mb R^n.\] {Passing to the limit $\e\to0$ in the last relation,  for sufficiently large $r>  R_0$,  it follows that
		\begin{equation}\label{ineqs1}
			\rho(r)w_\ga(x)\leq 	u(x), \;\; \mbox{for  all } \ x\in \mb R^n.
		\end{equation}
		Thus, \eqref{ineqs1} implies that  \begin{align}\label{gbl}
			\rho(r) = \ds\inf_{\mb R^n} \frac {u} {w_\ga}=:C^*,
		\end{align} where $C^*>0$ is a positive constant, independent of  $r$.
		Now for every sufficiently large $r>2R_0$, define the function
		\begin{equation*}
			\varphi_r(x) := u_r(x) - C^* w_\ga(rx), 
		\end{equation*} where $C^*$ is as in \eqref{gbl}.
		Observe that by \EQ{gbl}, we get  $\varphi_r (x)\geq 0$ in $\R^n$. 
		Also, using the  definition of $\varphi_r$, \eqref{ur} and \eqref{ri}, for sufficiently large $r>2 R_0$, we get
		\begin{align}\label{gvg}
			(-\Delta)^s \varphi_r(x) &= (-\De)^su_r(x)-C^*(-\De)^sw_\ga(rx)\geq r^{2 s}f(u_{r}), \text{  for all   } x\in\mb R^n\setminus B_{\frac12}.
		\end{align}
		Therefore, for every sufficiently large $r>2 R_0$, applying Proposition \ref{kslap} with $A=A_r$ and $h(x):= r^{2s}f(u_r(x))$ and using  \eqref{lkj}, we deduce that
		\begin{equation}\label{e1}
			\inf_{B_2\setminus B_1} \varphi_r \geq \frac 12 \bar c |B_2\setminus B_1| \min_{t\in \left[ m(r),\bar C m(r)\right]}f(t)  \geq a r^{
				\sigma^*}, 
		\end{equation} where the  positive constants $\bar c$ and $a$ are as defined in Proposition \ref {kslap} and in \eqref{lkj}, respectively, which are independent of $r$.
		Thus, for every sufficiently large $r>2R_0$ and  for  all $x\in B_{2}\setminus B_{1},$ from \eqref{e1} and the definition of $w_\ga$, we deduce
		\begin{align}\label{pn}
			u_r(x)-C^*w_\ga(rx)&\geq a|rx|^{\sigma^*}\geq \frac{a}{1+\mu}w_\ga(rx).
		\end{align}	 This implies that 
		there exists some large $r_1>2R_0$ such that {for every } $\ r \geq r_1$,  we have
		\begin{align}\label{e2}
			\left( C^*+\frac{a}{1+\mu}\right) w_\ga(x)\leq u(x), \text{  for all  } x\in B_{2r}\setminus B_{r}.
		\end{align} Thus, combining \eqref{e2} and \eqref{gbl}, it yields that \[C^*+\frac{a}{1+\mu}\leq \inf_{ B_{2r} \setminus B_{r}}\frac{u(x)}{w_\ga(x)}=\rho(r)=C^*,\] which is absurd. 
		Thus, the proof of the theorem is complete.}\hfill{\QED}

	
	\section{Data Availability}
	Data sharing is not applicable to this article as no data sets were generated or analyzed during the current study.
	\section{Conflict of Interest}
	On behalf of all authors, the corresponding author states that there is no conflict of interest.
	
	\section{Acknowledgment} We sincerely thank the anonymous reviewers for their  comments and valuable suggestions.\\ 
	R. Biswas was supported by ANID FONDECYT Postdoctoral Project No. 3230657 and A. Quaas was partially supported by ANID FONDECYT Grant No. 1231585.
	\linespread{0.5}
	
\end{document}